\documentclass[a4paper,10pt]{article}
\usepackage{graphicx}
\usepackage{amsmath,amssymb,amsfonts,amsthm}
\usepackage{url, enumerate,anysize}
\usepackage{lscape}
\usepackage[subnum]{cases}

\usepackage[usenames,dvipsnames,svgnames,table]{xcolor}
\usepackage{color}
\usepackage[colorlinks=true,linkcolor=black, citecolor=blue, urlcolor=blue]{hyperref}
\usepackage{graphicx}
\usepackage{tikz}
\usepackage[titletoc]{appendix}

\newtheorem{thm}{Theorem}[section]

\newtheorem{rmk}{Remark}

\usepackage[usenames,dvipsnames,svgnames,table]{xcolor}

\newtheorem*{theorem-non}{Conjecture}
\newtheorem{clm}{Claim}
\usepackage{algorithm}
\usepackage{algorithmic}

\title{On a Conjecture for Dynamic Priority Queues and Nash Equilibrium  for Quality of Service Sensitive Markets}
\author{Manu K. Gupta$^1$\footnote{Emails: Manu K. Gupta (manu-kumar.gupta@irit.fr) and N. Hemachandra (nh@iitb.ac.in)} and N. Hemachandra$^2$\\ $^1$IRIT, 2 rue C. Camichel, Toulouse, France \\$^2$Industrial Engineering and Operations Research, IIT Bombay, India}

\marginsize{1.5cm}{1.5cm}{1.5cm}{1.5cm}
\begin{document}
\maketitle 

\begin{abstract}
Many economic transactions, including those of online markets, have a time lag between the start and end times of transactions. Customers need to wait for completion of their transaction (order fulfillment) and hence are also interested in their waiting time as a Quality of Service (QoS) attribute. So, they factor this QoS in the demand they offer to the firm (service-provider) and some customers  (user-set) would be willing to pay for shorter waiting times. On the other hand, such waiting times depend on the demand user-set offers to the service-provider. We model the above economic-QoS strategic interaction between service-provider and user-set under a fairly generic scheduling framework as a non-cooperative constrained game.

We use an existing joint pricing and scheduling model. An optimal solution to this joint pricing and scheduling problem was guaranteed by a finite step algorithm subject to a conjecture. We first settle this conjecture based on queuing and optimization arguments and discuss its implications on the above game. We show that a continuum of Nash equilibria (NE) exists and it can be computed easily using constrained best response dynamics. Revenue maximal NE is identified by above finite step algorithm. We illustrate how both players can benefit at such revenue maximal NE by identifying suitable operational decisions, i.e., by choosing an appropriate game along the theme of pricing and revenue management. 

\end{abstract}


\section{Introduction} 
{Services and their Quality of Service (QoS) attribute form a major portion of web based Internet economy. The  offered services involve downloading/installation of an application software (app) and those where customers need to wait for some amount of time for their request/transaction to be completed by the service-provider (firm). We focus in a scenario where customers are willing to wait and apart from the price of the service, such customers are also interested in their waiting time as a QoS attribute. Such economic activities have other interesting features; the prominent being the finite resource capacity of the service-provider, the key reason why customers need to wait. Some other consequences include pricing of such services, scheduling of these finite capacity resources, etc. Our theme is the behavior of equilibria in markets where customers are sensitive to the price and service levels they receive which is determined by suitably defined QoS measures.

Above strategic and economic interaction arises in many fundamental resource allocation problems. We mention below two motivating examples. First one from the High Performance Computing (HPC) facility and another one in an online retail sector. In such scenarios, customers' requests form queues during their interactions with the service-provider.

The HPC center at the National Renewable Energy Laboratory (NREL) provides high-speed and large-scale computer processing facility. The NREL HPC system is one of the largest HPC systems in the world dedicated to advancing renewable energy and energy efficiency technologies \cite{Users}. The computing power of this high-end technology being scarce, jobs are queued up and will be completed eventually. Users are charged certain price for using this facility. However, users can reduce the queue waiting time by paying more to HPC facility. Jobs are given (non-preemptive strict) priority if they pay twice the normal rate \cite{Queues}. 

Another motivating example is the online retail stores for fast moving consumer goods, electronic gadgets, durables, etc., deliver goods within some pre-specified (standard) delivery time. Customers in many cases can pay a premium charge in return for faster delivery time; we call this class of customers secondary class of customers while the primary class of customers are those who choose to wait for longer time in return for no additional payments -- primary class because they usually form a large portion of the overall customer base. In such situations, the retail stores have to carefully manage their finite resources like warehouses, logistics, etc. Note that such stores while quoting the extra charge as well as the faster delivery time need to make sure that the goods are delivered to the primary customers within the standard delivery time. Separate Service Level Agreements (SLAs) between the online store and two classes of customers formalize such quoted (faster) delivery times and (extra) prices. Differential pricing and scheduling policies of queues play crucial roles in such interactions. Also, mean waiting time for each class of customers is a suitable QoS measure here. }

{{We study a model of above strategic interaction in the presence of congestion.} We consider HPC or the online retail seller as service-provider or firm. We consider the setup wherein the firm is already offering service to some existing (primary class of) customers, with a pre-specified price, arrival rate and service level as part of the service level agreement, $SLA_p$ it has with this set of customers. The firm has some spare resource capacity after serving the primary class of customers and is considering to offer this spare resource to another (secondary) class of customers, which we call user-set. It intends to maximize its revenue rate from the secondary class of customers.  We assume that the primary customers offer a stable demand and they don't switch to secondary class; this assumption is reasonable at an aggregate level as there is always a certain portion of customers who are willing to wait rather than pay price for faster delivery, i.e., we assume that there is a market segmentation. From the user-set point of view two important aspects are the price of the service and the waiting time incurred. In fact, some customers will be willing to pay more for this service to have lower mean waiting times. For example Amazon offers to Amazon Prime members with the default free two-day delivery -- it offers a small incentive (e.g., 1 digital credit) to a Prime member to opt for the standard one-week delivery instead. The user-set intends to maximize the rate of customers that the firm is able to offer this particular service. } 

Specifically, we consider the framework in which the mean demand the user-set offers to the service-provider is negatively linear in the (extra) price and the mean waiting time of a typical customer -- the linear demand setup capturing the sensitivity of the user-set to the price and offered QoS level. The service-provider, for a given demand distribution, has to decide on the price of the service as well as the mean waiting time it can guarantee to the user-set constrained by the $SLA_p$ it already had with the primary class of customers.  Note that the mean waiting time of customers that the service-provider can guarantee depends on the mean demand that the user-set offers to it, which in turn depends on the mean waiting time they are guaranteed; hence this interaction is strategic and the natural question would be the existence of equilibrium. Further, the existing $SLA_p$ with the primary class of customers leads to a constrained strategic interaction as illustrated in Figure \ref{new-fig}. All notations of this figure are elaborated in Section \ref{model_des}.\vspace{-0.2cm}
\begin{figure}[h]\centering 
\resizebox{0.45 \textwidth}{!}{\ifx\du\undefined
  \newlength{\du}
\fi
\setlength{\du}{15\unitlength}
\begin{tikzpicture}[scale = 0.85]
\pgftransformxscale{1.000000}
\pgftransformyscale{-1.000000}
\definecolor{dialinecolor}{rgb}{0.000000, 0.000000, 0.000000}
\pgfsetstrokecolor{dialinecolor}
\definecolor{dialinecolor}{rgb}{1.000000, 1.000000, 1.000000}
\pgfsetfillcolor{dialinecolor}
\pgfsetlinewidth{0.100000\du}
\pgfsetdash{}{0pt}
\pgfsetdash{}{0pt}
\pgfsetbuttcap
\pgfsetmiterjoin
\pgfsetlinewidth{0.100000\du}
\pgfsetbuttcap
\pgfsetmiterjoin
\pgfsetdash{}{0pt}
\definecolor{dialinecolor}{rgb}{1.000000, 1.000000, 1.000000}
\pgfsetfillcolor{dialinecolor}
\pgfpathmoveto{\pgfpoint{2.428571\du}{-15.120290\du}}
\pgfpathlineto{\pgfpoint{11.321429\du}{-15.120290\du}}
\pgfpathlineto{\pgfpoint{13.100000\du}{-10.070290\du}}
\pgfpathlineto{\pgfpoint{11.321429\du}{-5.020290\du}}
\pgfpathlineto{\pgfpoint{2.428571\du}{-5.020290\du}}
\pgfpathlineto{\pgfpoint{0.650000\du}{-10.070290\du}}
\pgfpathlineto{\pgfpoint{2.428571\du}{-15.120290\du}}
\node[anchor=west] at (3.25000\du,-14.0000\du){Service-Provider};
\node[anchor=west] at (4.25000\du,-13.0000\du){as a Queue};
\node[anchor=west] at (3.750000\du,-11.0000\du){(Pricing and};
\node[anchor=west] at (2.750000\du,-10.2500\du){ operational decisions)};
\node[anchor=west] at (1.6000\du,-9.0000\du){Price $\theta$, Scheduling policy $\beta$ and};
\node[anchor=west] at (1.50\du,-8.25000\du){Mean waiting time bound $S_s$};
\node[anchor=west] at (4.7500\du,-7.5000\du){based on };
\node[anchor=west] at (2.750000\du,-6.5000\du){demand with mean $\lambda_s$};
\pgfusepath{fill}
\definecolor{dialinecolor}{rgb}{0.000000, 0.000000, 0.000000}
\pgfsetstrokecolor{dialinecolor}
\pgfpathmoveto{\pgfpoint{2.428571\du}{-15.120290\du}}
\pgfpathlineto{\pgfpoint{11.321429\du}{-15.120290\du}}
\pgfpathlineto{\pgfpoint{13.100000\du}{-10.070290\du}}
\pgfpathlineto{\pgfpoint{11.321429\du}{-5.020290\du}}
\pgfpathlineto{\pgfpoint{2.428571\du}{-5.020290\du}}
\pgfpathlineto{\pgfpoint{0.650000\du}{-10.070290\du}}
\pgfpathlineto{\pgfpoint{2.428571\du}{-15.120290\du}}
\pgfusepath{stroke}
\definecolor{dialinecolor}{rgb}{0.000000, 0.000000, 0.000000}
\pgfsetstrokecolor{dialinecolor}
\node[anchor=west] at (12.000\du,-17.00\du){{\large Demand with mean $\lambda_s$}};
\node[anchor=west] at (12.000\du,-3.500\du){{\large Price $\theta$ and Quality of Service $S_s$}};
\node at (6.875000\du,-9.870290\du){};
\pgfsetlinewidth{0.100000\du}
\pgfsetdash{}{0pt}
\pgfsetdash{}{0pt}
\pgfsetbuttcap
\pgfsetmiterjoin
\pgfsetlinewidth{0.100000\du}
\pgfsetbuttcap
\pgfsetmiterjoin
\pgfsetdash{}{0pt}
\definecolor{dialinecolor}{rgb}{1.000000, 1.000000, 1.000000}
\pgfsetfillcolor{dialinecolor}
\pgfpathmoveto{\pgfpoint{21.364286\du}{-13.000000\du}}
\pgfpathlineto{\pgfpoint{27.435714\du}{-13.000000\du}}
\pgfpathlineto{\pgfpoint{28.650000\du}{-10.025000\du}}
\pgfpathlineto{\pgfpoint{27.435714\du}{-7.050000\du}}
\pgfpathlineto{\pgfpoint{21.364286\du}{-7.050000\du}}
\pgfpathlineto{\pgfpoint{20.150000\du}{-10.025000\du}}
\pgfpathlineto{\pgfpoint{21.364286\du}{-13.000000\du}}
\node[anchor=west] at (21.500\du,-12.00\du){User-set offers};
\node[anchor=west] at (21.500\du,-11.2500\du){demand with mean};
\node[anchor=west] at (21.5000\du,-10.2500\du){$\lambda_s \leq a -b\theta -cS_s$};
\node[anchor=west] at (21.5000\du,-9.2500\du){($b$ and $c$ are };
\node[anchor=west] at (21.000\du,-8.2500\du){demand elasticities)};
\pgfusepath{fill}
\definecolor{dialinecolor}{rgb}{0.000000, 0.000000, 0.000000}
\pgfsetstrokecolor{dialinecolor}
\pgfpathmoveto{\pgfpoint{21.364286\du}{-13.000000\du}}
\pgfpathlineto{\pgfpoint{27.435714\du}{-13.000000\du}}
\pgfpathlineto{\pgfpoint{28.650000\du}{-10.025000\du}}
\pgfpathlineto{\pgfpoint{27.435714\du}{-7.050000\du}}
\pgfpathlineto{\pgfpoint{21.364286\du}{-7.050000\du}}
\pgfpathlineto{\pgfpoint{20.150000\du}{-10.025000\du}}
\pgfpathlineto{\pgfpoint{21.364286\du}{-13.000000\du}}
\pgfusepath{stroke}
\definecolor{dialinecolor}{rgb}{0.000000, 0.000000, 0.000000}
\pgfsetstrokecolor{dialinecolor}
\node at (24.400000\du,-9.825000\du){};
\pgfsetlinewidth{0.100000\du}
\pgfsetdash{}{0pt}
\pgfsetdash{}{0pt}
\pgfsetbuttcap
{
\definecolor{dialinecolor}{rgb}{0.000000, 0.000000, 0.000000}
\pgfsetfillcolor{dialinecolor}
\pgfsetarrowsend{to}
\definecolor{dialinecolor}{rgb}{0.000000, 0.000000, 0.000000}
\pgfsetstrokecolor{dialinecolor}
\pgfpathmoveto{\pgfpoint{21.364384\du}{-12.999914\du}}
\pgfpatharc{316}{249}{9.265651\du and 9.265651\du}
\pgfusepath{stroke}
}
\pgfsetlinewidth{0.100000\du}
\pgfsetdash{}{0pt}
\pgfsetdash{}{0pt}
\pgfsetbuttcap
{
\definecolor{dialinecolor}{rgb}{0.000000, 0.000000, 0.000000}
\pgfsetfillcolor{dialinecolor}
\pgfsetarrowsend{to}
\definecolor{dialinecolor}{rgb}{0.000000, 0.000000, 0.000000}
\pgfsetstrokecolor{dialinecolor}
\pgfpathmoveto{\pgfpoint{11.320652\du}{-5.020520\du}}
\pgfpatharc{107}{51}{10.923280\du and 10.923280\du}
\pgfusepath{stroke}
}
\end{tikzpicture}}\vspace{-0.4cm}
 \caption{Schematic view of service-provider (as a queue) and user-set (customers) interaction.} 
 \label{new-fig}
\vspace{-0.3cm}
\end{figure}
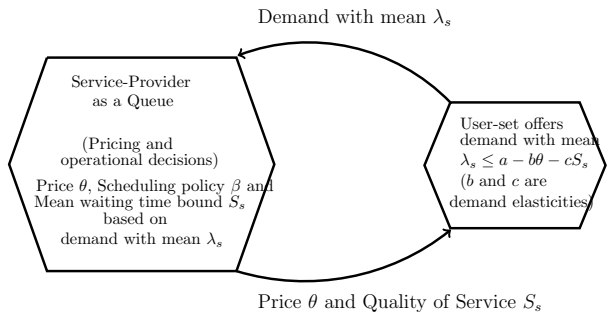

We model the above interaction under a fairly generic scheduling scheme, i.e., under the set of all non-preemptive,  non-anticipative and work conserving policies. We exploit the \textit{completeness} feature of delay dependent priority (DDP) scheme, introduced by Kleinrock \cite{Kleinrock1964}, to pose the problem in an existing framework of a joint pricing and scheduling model introduced in  \cite{Sudhir_standard_style}. \textit{Completeness} of a dynamic priority implies that it achieves all possible achievable vectors of mean waiting time (obtained by any non pre-emptive, non anticipative and work conserving scheduling scheme) for a suitable value of the priority parameter. See \cite{complete} and \cite{coffman} for detailed discussion on completeness and \cite{federgruen}, \cite{manu_vtool} for the completeness feature of different dynamic priority scheduling schemes including delay dependent priority discipline.

 An optimal solution to this joint pricing and scheduling problem was guaranteed by a finite step algorithm subject to a conjecture (see \cite{Sudhir_standard_style}). We first settle this conjecture based on queuing and optimization arguments and discuss its implications on a non-standard game for above strategic interaction.

{We show that a continuum of Nash equilibria (NE) exists and it can be computed easily using constrained best response dynamics. The constrained best response dynamics indeed finds a NE in one step for a given game. Further, revenue maximal NE in this continuum of NE is identified by above finite step algorithm.} We illustrate with numerical instances that both players can benefit at such revenue maximal NE by identifying suitable operational decisions, i.e., by choosing an appropriate game.  Several interesting game theoretic issues are under investigation such as impact of non-linear market in strategic interaction, intervention or mechanism design for the market so that the revenue maximal NE is achieved. We now briefly present some important and related work in this field.   
\subsection{Related literature}
In literature, Naor \cite{Naor} was first to study the issues of pricing and congestion control in an observable $M/M/1/FCFS$ queue. A huge literature has evolved since then. A comprehensive survey on topics related to queue pricing can be seen in \cite{hassin16} and \cite{hassin2003queue}. A similar model where a resource is shared by two different classes of customers is studied in \cite{hall2009static}. This model assumes that the demand is sensitive to unit price alone and focuses on dynamic pricing policies which depend on the production system (queue) status. A related model of optimal pricing with pre-emptive scheduling has been recently analyzed in \cite{gupta2017optimal}. On the other hand, the study of equilibrium behaviour of customers and servers in queueing system has long history and researchers have developed a broad literature in this field by now (see \cite{burnetas2013customer, doncel2014resource, hassin16, hassin2002nash, hassin2003queue, haviv2009strategic, nguyen2013constraint} and references there in).{ In above literature, often customers (or customer classes) are considered as player whereas we consider market as a player by aggregating all customers into a market segment.}

{ The delay dependent priority queue has been recently called as accumulating priority queue and analysis is advanced to obtain the distribution of waiting time (see \cite{li2017nonlinear},  \cite{stanford2014waiting}). Also related to accumulating priority queue, the issues of pricing has been dealt with in \cite{haviv2016strategic} and scheduling in \cite{li2016multi}. Pricing problem in \cite{haviv2016strategic} is formulated as a non-cooperative game in which
customers purchase priority coefficients with the goal of reducing waiting costs in exchange. It is argued that if waiting costs are linear then the Nash equilibrium is in pure strategies. Authors in \cite{li2016multi} deal with the scheduling of multi-class  multi-server queue where a job is scheduled with one of the idle server using $r$-dispatch policy introduced in \cite{doroudi2011dispatching}. Clearly, these settings are different from this paper. 

\section{System Description}\label{model_des}
 The model addresses the question of pricing the surplus server capacity of a stable $M/G/1$ queue for a new class of customers that is sensitive to its QoS (mean waiting time) and price. The primary class of customers arrive according to Poisson arrival process with rate $\lambda_p$. $S_p$, the desired guarantee (or bound) on the mean waiting time of the primary class of customers, indicates the service level offered. The service time of customers is independent and identically distributed with mean $1/\mu$ and variance $\sigma^2$, irrespective of customer class. A schematic view of the model is shown in Figure \ref{fig:sks}. The objective is to determine the promised {{}bound} on the mean waiting time of a secondary class of customers, $S_s$ and their unit admission price $\theta$ so as to maximize the revenue generated by the system, while constrained by primary class service levels.  The secondary class of customers arrive according to an independent Poisson arrival process with rate $\lambda_s$, which is linearly dependent on price $\theta$ and QoS $S_s$:  $\lambda_s(\theta,S_s)=a - b\theta - cS_s$, where $a,~b,~c$ are positive constants driven by the market; $a$ is the maximal arrival rate offered by market, $b$ and $c$ are the elasticity of unit price and {{}assured} service level respectively.

{Before describing the model, we motivate two important assumptions of our system setup. First, that of linear demand function. We assume linear relationship for simplicity of analysis and it has been previously used in queuing literature (see, for example \cite{demand1}, \cite{demand2} and comprehensive references for the subsequent work in \cite{hassin16}). Linear production functions are also often used in economics literature \cite{baijmol1972economic}. Our another assumption is related to the fact that queue operator cannot optimize the primary class and customers cannot switch between classes.  
We list below few instances where such model is relevant.

One such model is where primary customers are inhouse customers
of a firm and hence their arrival rate cannot be controlled (optimized). Suppose the firm observes that some of its capacity and investment is underutilized and that there is a market for its capacity
and capability. Thus, the firm can `lease its facilities' to secondary customers without affecting the service level of existing/primary (inhouse) customers. Primary customers are inhouse customers of service provider and they sign a long term contract of receiving service at pre-specified service level. These customers are not allowed to break the contract and participate in new service level offered and price charged by service provider. Hence switches between classes is not possible.

 We explain the above behavior in the context of mobile network operators. In United States, the major network operators as of now are AT\&T, Sprint, T-Mobile and
Verizon (see \cite{operators}). These are the big four and they offer their own mobile phone plans and contracts
and have a primary customer base. There are virtual operators who lease the service from these big
four and have customers of their own. For example, Boost Mobile has customers of its own, but it uses the Sprint network. In such situations, above joint pricing and scheduling model gives a guideline
to price virtual operator's customers. Other firms could be a container depot, a large manufacturing plant, etc. where long term contracts are feasible and above pricing model is relevant. In above examples, customers can't change their class. Such assumptions have also been previously used in literature (see \cite{gupta2017optimal}, \cite{Sudhir_standard_style} and references therein). However, example of the models where customers are allowed to change classes are DiffServ model (see \cite{tandra2004diffserv}).

}

Let $\mathcal{F}$ be the set of all non pre-emptive, non anticipative and work conserving scheduling schemes. The mean waiting time of primary and secondary class customers depend on the scheduling scheme $\pi \in \mathcal{F}$. Let $W_{p}(\lambda_{s}, \pi)$ and $W_{s}(\lambda_{s}, \pi)$ be the mean waiting times of primary and secondary customers, when the arrival rate of secondary jobs is $\lambda_{s}$ and scheduling policy is $\pi \in \mathcal{F}$. 
Now, select a suitable pair of pricing parameters $\theta$ and $S_{s}$ for the secondary class customers, a queue scheduling policy $\pi \in \mathcal{F}$ and an appropriate admission rate for the secondary class customers $\lambda_{s}$, that will maximize the expected revenue from their inclusion, while ensuring that the mean waiting time to the primary class customers does not exceed a given quantity $S_p$. Thus, the revenue maximization problem, $\mathcal{P}$, is: 
\begin{eqnarray}
\mbox{\textbf{$\mathcal{P}$:}\space}{\max_{\lambda_s,\pi, S_s , \theta}~} \theta\lambda_s
\end{eqnarray}
\begin{eqnarray}Subject~to:~~ 
W_p(\lambda_s,\pi)\leq S_p,\label{Pri_Qos}\\
W_s(\lambda_s,\pi)\leq S_s, \label{Sec_Qos}\\
\lambda_s \leq \mu - \lambda_p, \label{Sys_sta}\\
\lambda_s \leq a-b\theta-cS_s, \label{Dem}\\
\lambda_s,\theta,S_s\geq 0,\pi \in \mathcal{F}. 
\end{eqnarray}
Constraint (\ref{Pri_Qos}) and (\ref{Sec_Qos}) ensure the service level for the jobs of primary and secondary class respectively. Constraint (\ref{Sys_sta}) is queue stability constraint. {{}
Note that the constraint (\ref{Sys_sta}) will remain non-binding at the optimal solution to ensure finite QoS (mean waiting time) levels for both classes of customers.} Constraint (\ref{Dem}) ensures that the mean arrival rate of secondary class customers should not exceed the demand generated
by charged price $\theta$ and offered service level $S_s$.

{{}It is common in literature to assume that demand for service is sensitive to price and the expected delay (see, for example \cite{demand1}, \cite{demand2} and comprehensive references for the subsequent work in \cite{hassin16}). If one assumes such demand, i.e., $\lambda_s(\theta,\pi)=a - b\theta - cW_s(\lambda_s, \pi)$, then constraint (\ref{Sec_Qos}) is tight and P1 gives the optimal pricing and priority regime for this demand function. So the results of this paper also provide conditions for the optimal solutions of both demand models to coincide.
}

\label{sec:sysdes}
\begin{figure}[h]\centering 
\resizebox{0.485 \textwidth}{!}{\input{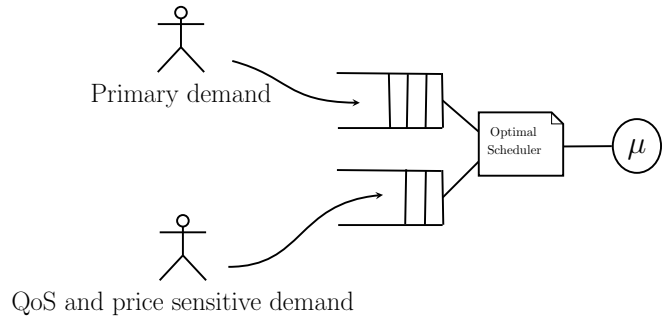}}
 \caption{Schematic view of the model } 
 \label{fig:sks}
\end{figure}
Delay dependent priority queue discipline is identified as a \textit{complete} class of schedulers for mean waiting time performance measure in \cite{federgruen}. Completeness of a scheduling discipline implies that all possible vectors of mean waiting time under any non pre-emptive, non anticipative and work conserving policy can be achieved by suitably chosen parameter from a complete scheduling discipline. Thus, optimizing over $\pi \in \mathcal{F}$ is equivalent to the optimization over DDP scheme. We now briefly explain the mechanism of DDP.

}
In DDP scheme, customers are scheduled according to their delay in queue. Mathematically, the instantaneous priority at time $t$ of class $c$ customer that arrived at time $T_{c}$ is calculated as $q_{c}(t):= (t - T_{c})b_{c}$ for some positive weight $b_{c}$ associated with class $c$. Let $c \in \{p, s\}$ so that $b_{p}$ and $b_{s}$ refer to the weights associated with primary and secondary classes, respectively. At each service completion, the server chooses the next job with the highest instantaneous priority $q_{c}(\cdot)$, $c \in \{p,s\}$. The steady state mean waiting time of each class of customers, derived by Kleinrock \cite{Kleinrock1964}, depends on the ratio of weights ($b_i$) given to each class. Let $\beta:=b_s/b_p$. Note that $\beta=0$ corresponds to static high priority to primary class customers, $\beta=1$ is the global First Come First Serve (FCFS) queuing discipline across classes and $\beta = \infty$ corresponds to static high priority to secondary class customers. Let $W_{p}(\lambda_{s}, \beta)$ and $W_{s}(\lambda_{s}, \beta)$ be the mean waiting times of primary and secondary customers, when the arrival rate of secondary jobs is $\lambda_{s}$ and queue management parameter is $\beta$.

In the view of completeness of DDP,  the revenue maximization problem, $\mathcal{P}$, can be equivalently written as: 
\begin{eqnarray}
\mbox{\textbf{P0:}\space}{\max_{\lambda_s,\beta, S_s , \theta}~} \theta\lambda_s
\end{eqnarray}
\begin{eqnarray}Subject~to:~~ 
W_p(\lambda_s,\beta)\leq S_p, \label{betaPri_Qos}\\
W_s(\lambda_s,\beta)\leq S_s, \label{betaSec_Qos}\\
\lambda_s \leq \mu - \lambda_p, \label{betaSys_sta}\\
\lambda_s \leq a-b\theta-cS_s, \label{betaDem}\\
\lambda_s,\theta,S_s\geq 0,~ \beta\in \{\mathbb{R}^+\cup\infty\}. 
\end{eqnarray}

Optimal solution to problem \textbf{P0} is characterized in \cite{Sudhir_standard_style}. The non-convex problem \textbf{P0} can be presented as following non-convex constrained optimization problem P1 (since constraints (\ref{betaSec_Qos}) and (\ref{betaDem}) are binding at optimality \cite{Sudhir_standard_style})
\begin{eqnarray}\label{objective1}
\mbox{\textbf{P1:}\space} \max_{\lambda_s,\beta}~ \dfrac{1}{b}\left( a\lambda_s -\lambda_s^2 -c \lambda_s W_s(\lambda_s,\beta)\right)  
\end{eqnarray}
\begin{eqnarray}Subject~to:~~ 
W_p(\lambda_s,\beta) \leq S_p, \label{eqn:cons11}\\
\lambda_s \leq \mu - \lambda_p,\\
\lambda_s ,\beta \geq  0. 
\end{eqnarray}
Once the optimal secondary class mean arrival rate $\lambda_{s}^{*}$ and {{}the} optimal queue discipline management parameter $\beta^{*}$ are obtained, the optimal admission price $\theta^{*}$ and {{}the} optimal assured service level to secondary class $S_{s}^{*}$ can be computed using $S_{s}^{*} = W_{s}(\lambda_{s}^{*} ,\beta^{*})$ and $\lambda_{s}^{*} = a-b\theta^{*} - cS_{s}^{*}$.
 
Note that above optimization problem P1 considers only finite values of $\beta$, though $\beta = \infty$ is also a valid decision variable as it corresponds to a static high priority to secondary class customers. Hence, the following one dimensional convex optimization problem, P2, wherein $\beta$ is set to $\infty$ in problem P1, is:
 \begin{equation}
 \mbox{\textbf{P2:}\space}  \max_{\lambda_s}~ \dfrac{1}{b}[a\lambda_s -\lambda_s^2-c\lambda_s\tilde{W_s}(\lambda_s)]
 \label{eqn:objective_P2}
 \end{equation}
 \begin{eqnarray}Subject~to:~~ 
 \tilde{W_p}(\lambda_s) \leq S_p, \label{eqn:constr29}\\
 \lambda_s \leq \mu- \lambda_p,\\
 \lambda_s \geq 0.
 \end{eqnarray}
 where $\tilde{W_p}(\lambda_s) = W_p(\lambda_s, \beta = \infty)$ and $\tilde{W_s}(\lambda_s) = W_s(\lambda_s, \beta = \infty)$. Above two optimization problems (P1 and P2) are analyzed for their global optima, and their optimal values are compared in \cite{Sudhir_standard_style} to give a solution to P0 via a finite step algorithm. {{}The optimal solutions of P1 and P2 are derived using Karush-Kuhn-Tucker (KKT) conditions.}

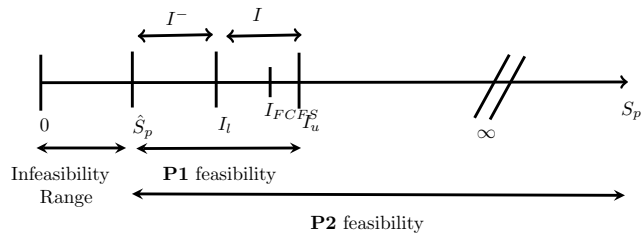
\begin{figure}[ht]
\centering
\resizebox{.48\textwidth}{!}{
\ifx\du\undefined
  \newlength{\du}
\fi
\setlength{\du}{15\unitlength}
\begin{tikzpicture}
\pgftransformxscale{1.000000}
\pgftransformyscale{-1.000000}
\definecolor{dialinecolor}{rgb}{0.000000, 0.000000, 0.000000}
\pgfsetstrokecolor{dialinecolor}
\definecolor{dialinecolor}{rgb}{1.000000, 1.000000, 1.000000}
\pgfsetfillcolor{dialinecolor}
\pgfsetlinewidth{0.100000\du}
\pgfsetdash{}{0pt}
\pgfsetdash{}{0pt}
\pgfsetbuttcap
{
\definecolor{dialinecolor}{rgb}{0.000000, 0.000000, 0.000000}
\pgfsetfillcolor{dialinecolor}
\definecolor{dialinecolor}{rgb}{0.000000, 0.000000, 0.000000}
\pgfsetstrokecolor{dialinecolor}
\draw (19.950000\du,9.100000\du)--(19.950000\du,11.050000\du);
}
\pgfsetlinewidth{0.100000\du}
\pgfsetdash{}{0pt}
\pgfsetdash{}{0pt}
\pgfsetbuttcap
{
\definecolor{dialinecolor}{rgb}{0.000000, 0.000000, 0.000000}
\pgfsetfillcolor{dialinecolor}
\definecolor{dialinecolor}{rgb}{0.000000, 0.000000, 0.000000}
\pgfsetstrokecolor{dialinecolor}
\draw[,,->] (19.950000\du,10.075000\du)--(40.250000\du,10.050000\du);
}
\pgfsetlinewidth{0.100000\du}
\pgfsetdash{}{0pt}
\pgfsetdash{}{0pt}
\pgfsetbuttcap
{
\definecolor{dialinecolor}{rgb}{0.000000, 0.000000, 0.000000}
\pgfsetfillcolor{dialinecolor}
\definecolor{dialinecolor}{rgb}{0.000000, 0.000000, 0.000000}
\pgfsetstrokecolor{dialinecolor}
\draw (23.115000\du,8.990000\du)--(23.115000\du,10.940000\du);
}
\pgfsetlinewidth{0.100000\du}
\pgfsetdash{}{0pt}
\pgfsetdash{}{0pt}
\pgfsetbuttcap
{
\definecolor{dialinecolor}{rgb}{0.000000, 0.000000, 0.000000}
\pgfsetfillcolor{dialinecolor}
\definecolor{dialinecolor}{rgb}{0.000000, 0.000000, 0.000000}
\pgfsetstrokecolor{dialinecolor}
\draw (26.015000\du,8.990000\du)--(26.015000\du,10.940000\du);
}
\pgfsetlinewidth{0.100000\du}
\pgfsetdash{}{0pt}
\pgfsetdash{}{0pt}
\pgfsetbuttcap
{
\definecolor{dialinecolor}{rgb}{0.000000, 0.000000, 0.000000}
\pgfsetfillcolor{dialinecolor}
\definecolor{dialinecolor}{rgb}{0.000000, 0.000000, 0.000000}
\pgfsetstrokecolor{dialinecolor}
\draw (27.880000\du,9.530000\du)--(27.880000\du,10.580000\du);
}

{
\definecolor{dialinecolor}{rgb}{0.000000, 0.000000, 0.000000}
\pgfsetfillcolor{dialinecolor}
\definecolor{dialinecolor}{rgb}{0.000000, 0.000000, 0.000000}
\pgfsetstrokecolor{dialinecolor}
\draw (28.880000\du,9.030000\du)--(28.880000\du,10.980000\du);
}
\pgfsetlinewidth{0.100000\du}
\pgfsetdash{}{0pt}
\pgfsetdash{}{0pt}
\pgfsetbuttcap
\definecolor{dialinecolor}{rgb}{0.000000, 0.000000, 0.000000}
\pgfsetstrokecolor{dialinecolor}
\node[anchor=west] at (19.650000\du,11.550000\du){0};
\definecolor{dialinecolor}{rgb}{0.000000, 0.000000, 0.000000}
\pgfsetstrokecolor{dialinecolor}
\node[anchor=west] at (22.850000\du,11.600000\du){$\hat{S}_p$};
\definecolor{dialinecolor}{rgb}{0.000000, 0.000000, 0.000000}
\pgfsetstrokecolor{dialinecolor}
\node[anchor=west] at (25.800000\du,11.550000\du){$I_l$};

\definecolor{dialinecolor}{rgb}{0.000000, 0.000000, 0.000000}
\pgfsetstrokecolor{dialinecolor}
\node[anchor=west] at (27.4000\du,11.0485000\du){$I_{FCFS}$};

\definecolor{dialinecolor}{rgb}{0.000000, 0.000000, 0.000000}
\pgfsetstrokecolor{dialinecolor}
\node[anchor=west] at (28.715000\du,11.485000\du){$I_u$};
\definecolor{dialinecolor}{rgb}{0.000000, 0.000000, 0.000000}
\pgfsetstrokecolor{dialinecolor}
\node[anchor=west] at (39.750000\du,11.0550000\du){$S_p$};
\pgfsetlinewidth{0.100000\du}
\definecolor{dialinecolor}{rgb}{0.000000, 0.000000, 0.000000}
\pgfsetstrokecolor{dialinecolor}
\node[anchor=west] at (34.750000\du,11.8550000\du){$\infty$};
\pgfsetlinewidth{0.100000\du}
\pgfsetdash{}{0pt}
\pgfsetdash{}{0pt}
\pgfsetbuttcap
{
\definecolor{dialinecolor}{rgb}{0.000000, 0.000000, 0.000000}
\pgfsetfillcolor{dialinecolor}
\pgfsetarrowsstart{to}
\pgfsetarrowsend{to}
\definecolor{dialinecolor}{rgb}{0.000000, 0.000000, 0.000000}
\pgfsetstrokecolor{dialinecolor}
\draw (23.300000\du,8.450000\du)--(25.900000\du,8.400000\du);
}
\pgfsetlinewidth{0.100000\du}
\pgfsetdash{}{0pt}
\pgfsetdash{}{0pt}
\pgfsetbuttcap
{
\definecolor{dialinecolor}{rgb}{0.000000, 0.000000, 0.000000}
\pgfsetfillcolor{dialinecolor}
\pgfsetarrowsstart{to}
\pgfsetarrowsend{to}
\definecolor{dialinecolor}{rgb}{0.000000, 0.000000, 0.000000}
\pgfsetstrokecolor{dialinecolor}
\draw (26.365952\du,8.458116\du)--(28.965952\du,8.408116\du);
}
\definecolor{dialinecolor}{rgb}{0.000000, 0.000000, 0.000000}
\pgfsetstrokecolor{dialinecolor}
\node[anchor=west] at (24.000000\du,7.800000\du){$I^-$};
\definecolor{dialinecolor}{rgb}{0.000000, 0.000000, 0.000000}
\pgfsetstrokecolor{dialinecolor}
\node[anchor=west] at (27.00000\du,7.700000\du){$I$};
\pgfsetlinewidth{0.100000\du}
\pgfsetdash{}{0pt}
\pgfsetdash{}{0pt}
\pgfsetbuttcap
{
\definecolor{dialinecolor}{rgb}{0.000000, 0.000000, 0.000000}
\pgfsetfillcolor{dialinecolor}
\pgfsetarrowsstart{stealth}
\pgfsetarrowsend{stealth}
\definecolor{dialinecolor}{rgb}{0.000000, 0.000000, 0.000000}
\pgfsetstrokecolor{dialinecolor}
\draw (23.250000\du,12.370000\du)--(28.950000\du,12.350000\du);
}
\pgfsetlinewidth{0.100000\du}
\pgfsetdash{}{0pt}
\pgfsetdash{}{0pt}
\pgfsetbuttcap
{
\definecolor{dialinecolor}{rgb}{0.000000, 0.000000, 0.000000}
\pgfsetfillcolor{dialinecolor}
\pgfsetarrowsstart{stealth}
\pgfsetarrowsend{stealth}
\definecolor{dialinecolor}{rgb}{0.000000, 0.000000, 0.000000}
\pgfsetstrokecolor{dialinecolor}
\draw (23.100000\du,13.950000\du)--(40.150000\du,14.000000\du);
}
\pgfsetlinewidth{0.100000\du}
\pgfsetdash{}{0pt}
\pgfsetdash{}{0pt}
\pgfsetbuttcap
{
\definecolor{dialinecolor}{rgb}{0.000000, 0.000000, 0.000000}
\pgfsetfillcolor{dialinecolor}
\definecolor{dialinecolor}{rgb}{0.000000, 0.000000, 0.000000}
\pgfsetstrokecolor{dialinecolor}
\draw (36.150000\du,9.100000\du)--(34.950000\du,11.300000\du);
}
\pgfsetlinewidth{0.100000\du}
\pgfsetdash{}{0pt}
\pgfsetdash{}{0pt}
\pgfsetbuttcap
{
\definecolor{dialinecolor}{rgb}{0.000000, 0.000000, 0.000000}
\pgfsetfillcolor{dialinecolor}
\definecolor{dialinecolor}{rgb}{0.000000, 0.000000, 0.000000}
\pgfsetstrokecolor{dialinecolor}
\draw (36.684598\du,9.157837\du)--(35.484598\du,11.357837\du);
}
\definecolor{dialinecolor}{rgb}{0.000000, 0.000000, 0.000000}
\pgfsetstrokecolor{dialinecolor}
\node[anchor=west] at (23.900000\du,13.300000\du){\textbf{P1} feasibility};
\definecolor{dialinecolor}{rgb}{0.000000, 0.000000, 0.000000}
\pgfsetstrokecolor{dialinecolor}
\node[anchor=west] at (29.000000\du,15.000000\du){\textbf{P2} feasibility};
\pgfsetlinewidth{0.100000\du}
\pgfsetdash{}{0pt}
\pgfsetdash{}{0pt}
\pgfsetbuttcap
{
\definecolor{dialinecolor}{rgb}{0.000000, 0.000000, 0.000000}
\pgfsetfillcolor{dialinecolor}
\pgfsetarrowsstart{stealth}
\pgfsetarrowsend{stealth}
\definecolor{dialinecolor}{rgb}{0.000000, 0.000000, 0.000000}
\pgfsetstrokecolor{dialinecolor}
\draw (19.800000\du,12.350000\du)--(22.900000\du,12.350000\du);
}
\definecolor{dialinecolor}{rgb}{0.000000, 0.000000, 0.000000}
\pgfsetstrokecolor{dialinecolor}
\node[anchor=west] at (18.650000\du,13.300000\du){\text{Infeasibility}};
\definecolor{dialinecolor}{rgb}{0.000000, 0.000000, 0.000000}
\pgfsetstrokecolor{dialinecolor}
\node[anchor=west] at (19.650000\du,14.100000\du){Range};
\end{tikzpicture}}
\caption{Illustration for range of $S_p$ with optimal solutions coming from problem P1 and P2}\label{fig:range}
\end{figure}

{

Solution of these optimization problems are obtained in terms of different ranges of primary class service levels ($S_p$). For service level $S_p \le \hat{S}_p = \dfrac{\lambda_p \psi}{\mu(\mu-\lambda_p)}$, the problems are infeasible {{}(no secondary class customers can be accommodated) because the promised bound for primary customers cannot be achieved, even without considering secondary customer}. If $\dfrac{a}{c}>\dfrac{\lambda_p(2\mu-\lambda_p)}{\mu(\mu-\lambda_p)^2}\psi$ \text{ where } $\psi = \dfrac{1+\sigma^2\mu^2}{2}$ then both the optimization problems P1 and P2 have (global) optimal solutions for $S_{p} \in I^{-} \cup I$ (see Figure \ref{fig:range}), for suitably identified finite intervals $I^-\equiv (\hat{S}_p, I_l )$ and $I\equiv [I_l, I_u) $ (see Theorem 1 and 2 in \cite{Sudhir_standard_style}) {where $I_l$ and $I_u$ are:  
\begin{equation}
I_l = \frac{\lambda_1\psi}{\mu(\mu-\lambda_p)}\text{ and }I_u = \frac{\lambda_1\psi}{(\mu-\lambda_s^{(1)})(\mu-\lambda_1)}
\end{equation}
Here $\lambda_1 = \lambda_p + \lambda_s^{(1)}$ and $\lambda_s^{(1)}$ is the unique root of cubic $G(\lambda_s)$ in interval $(0,\mu-\lambda_p)$ where $G(\lambda_s)$ is
$$ 2\mu\lambda_s^3 -[c\psi+\mu(a+4\phi_0)]\lambda_s^2 +2\phi_0[c\psi + \mu(a+\phi_0)]\lambda_s-a\mu\phi_0^2+c\psi\lambda_p(\mu+\phi_0),\text{\hspace{0.7cm}}$$
with $\phi_0 = \mu-\lambda_p$. It follows from above definition of $I_l$ and $I_u$ that $\hat{S}_p < I_l < I_u$ and hence service level ranges $I^-$ and $I$ are non-empty intervals.} Problem P1 is undefined\footnote{{{} Note that problem P1 is defined for any non-negative but finite $\beta$. As $S_p \uparrow I_u$, $\beta \uparrow \infty$. Thus, problem P1 is undefined for $S_p \ge I_u$. }} for service level range  $S_p \ge I_u$ whereas both problems P1 and P2 are feasible for $S_p\in I^- \cup I$. Hence, one has to compare the optimal objectives of problem P1 and P2 to obtain the global optima of problem P0 in service level range $I^- \cup I$. It was proved in \cite{Sudhir_standard_style} that for $S_p \in I$, the optimal solution of P0 is given by P1, i.e., {{}the} optimal objective of P1, $O_1^*$, is more than that of P2, $O_2^*$, for $S_p \in I$. Further, based on computational evidences, following result was \textit{conjectured} in \cite{Sudhir_standard_style}.

\begin{theorem-non}[\cite{Sudhir_standard_style}]
For $S_p \in I^-$, the optimal solution of P0 is given by optimal solution of P1.
\end{theorem-non}
}

 {

{ We elaborate on the statement of above conjecture and comment that {{}the} conjecture cannot be settled straight forwardly. 
 \begin{figure}[h]
\centering
\resizebox{.4 \textwidth}{!}{
\ifx\du\undefined
  \newlength{\du}
\fi
\setlength{\du}{15\unitlength}
\begin{tikzpicture}
\pgftransformxscale{1.000000}
\pgftransformyscale{-1.000000}
\definecolor{dialinecolor}{rgb}{0.000000, 0.000000, 0.000000}
\pgfsetstrokecolor{dialinecolor}
\definecolor{dialinecolor}{rgb}{1.000000, 1.000000, 1.000000}
\pgfsetfillcolor{dialinecolor}
\pgfsetlinewidth{0.100000\du}
\pgfsetdash{}{0pt}
\pgfsetdash{}{0pt}
\pgfsetbuttcap
{
\definecolor{dialinecolor}{rgb}{0.000000, 0.000000, 0.000000}
\pgfsetfillcolor{dialinecolor}
\pgfsetarrowsend{stealth}
\definecolor{dialinecolor}{rgb}{0.000000, 0.000000, 0.000000}
\pgfsetstrokecolor{dialinecolor}
\draw (6.265000\du,15.535000\du)--(19.065000\du,15.485000\du);
}
\pgfsetlinewidth{0.100000\du}
\pgfsetdash{}{0pt}
\pgfsetdash{}{0pt}
\pgfsetbuttcap
{
\definecolor{dialinecolor}{rgb}{0.000000, 0.000000, 0.000000}
\pgfsetfillcolor{dialinecolor}
\pgfsetarrowsend{stealth}
\definecolor{dialinecolor}{rgb}{0.000000, 0.000000, 0.000000}
\pgfsetstrokecolor{dialinecolor}
\draw (6.315000\du,15.535000\du)--(6.315000\du,4.885000\du);
}
\pgfsetlinewidth{0.000000\du}
\pgfsetdash{{1.000000\du}{1.000000\du}}{0\du}
\pgfsetdash{{0.200000\du}{0.200000\du}}{0\du}
\pgfsetbuttcap
{
\definecolor{dialinecolor}{rgb}{0.000000, 0.000000, 0.000000}
\pgfsetfillcolor{dialinecolor}
\definecolor{dialinecolor}{rgb}{0.000000, 0.000000, 0.000000}
\pgfsetstrokecolor{dialinecolor}
\draw (10.415000\du,4.985000\du)--(10.465000\du,15.435000\du);
}
\pgfsetlinewidth{0.000000\du}
\pgfsetdash{{0.200000\du}{0.200000\du}}{0\du}
\pgfsetdash{{0.200000\du}{0.200000\du}}{0\du}
\pgfsetbuttcap
{
\definecolor{dialinecolor}{rgb}{0.000000, 0.000000, 0.000000}
\pgfsetfillcolor{dialinecolor}
\definecolor{dialinecolor}{rgb}{0.000000, 0.000000, 0.000000}
\pgfsetstrokecolor{dialinecolor}
\draw (17.230239\du,5.025239\du)--(17.280239\du,15.475239\du);
}
\pgfsetlinewidth{0.100000\du}
\pgfsetdash{}{0pt}
\pgfsetdash{}{0pt}
\pgfsetbuttcap
{
\definecolor{dialinecolor}{rgb}{0.000000, 0.000000, 0.000000}
\pgfsetfillcolor{dialinecolor}
\definecolor{dialinecolor}{rgb}{0.000000, 0.000000, 0.000000}
\pgfsetstrokecolor{dialinecolor}
\draw (10.415000\du,9.685000\du)--(17.215000\du,7.735000\du);
}
\pgfsetlinewidth{0.100000\du}
\pgfsetdash{}{0pt}
\pgfsetdash{}{0pt}
\pgfsetbuttcap
{
\definecolor{dialinecolor}{rgb}{0.000000, 0.000000, 0.000000}
\pgfsetfillcolor{dialinecolor}
\definecolor{dialinecolor}{rgb}{0.000000, 0.000000, 0.000000}
\pgfsetstrokecolor{dialinecolor}
\pgfpathmoveto{\pgfpoint{10.515279\du}{9.584628\du}}
\pgfpatharc{217}{184}{11.103314\du and 11.103314\du}
\pgfusepath{stroke}
}
\pgfsetlinewidth{0.100000\du}
\pgfsetdash{}{0pt}
\pgfsetdash{}{0pt}
\pgfsetmiterjoin
\pgfsetbuttcap
{
\definecolor{dialinecolor}{rgb}{0.000000, 0.000000, 0.000000}
\pgfsetfillcolor{dialinecolor}
\definecolor{dialinecolor}{rgb}{0.000000, 0.000000, 1.000000}
\pgfsetstrokecolor{dialinecolor}
\pgfpathmoveto{\pgfpoint{8.300000\du}{15.600000\du}}
\pgfpathcurveto{\pgfpoint{10.885000\du}{10.465000\du}}{\pgfpoint{4.800000\du}{12.600000\du}}{\pgfpoint{17.015000\du}{7.835000\du}}
\pgfusepath{stroke}
}
\pgfsetlinewidth{0.050000\du}
\pgfsetdash{}{0pt}
\pgfsetdash{}{0pt}
\pgfsetbuttcap
{
\definecolor{dialinecolor}{rgb}{0.000000, 0.000000, 0.000000}
\pgfsetfillcolor{dialinecolor}
\pgfsetarrowsstart{to}
\pgfsetarrowsend{to}
\definecolor{dialinecolor}{rgb}{0.000000, 0.000000, 0.000000}
\pgfsetstrokecolor{dialinecolor}
\draw (8.415000\du,16.885000\du)--(10.465000\du,16.885000\du);
}
\pgfsetlinewidth{0.050000\du}
\pgfsetdash{}{0pt}
\pgfsetdash{}{0pt}
\pgfsetbuttcap
{
\definecolor{dialinecolor}{rgb}{0.000000, 0.000000, 0.000000}
\pgfsetfillcolor{dialinecolor}
\pgfsetarrowsstart{to}
\pgfsetarrowsend{to}
\definecolor{dialinecolor}{rgb}{0.000000, 0.000000, 0.000000}
\pgfsetstrokecolor{dialinecolor}
\draw (10.865000\du,16.835000\du)--(17.515000\du,16.835000\du);
}
\definecolor{dialinecolor}{rgb}{0.000000, 0.000000, 0.000000}
\pgfsetstrokecolor{dialinecolor}
\node[anchor=west] at (7.65000\du,16.285000\du){$\hat{S}_p$};
\definecolor{dialinecolor}{rgb}{0.000000, 0.000000, 0.000000}
\pgfsetstrokecolor{dialinecolor}
\node[anchor=west] at (10.215000\du,16.285000\du){$I_l$};
\definecolor{dialinecolor}{rgb}{0.000000, 0.000000, 0.000000}
\pgfsetstrokecolor{dialinecolor}
\node[anchor=west] at (16.8065000\du,16.185000\du){$I_u$};
\definecolor{dialinecolor}{rgb}{0.000000, 0.000000, 0.000000}
\pgfsetstrokecolor{dialinecolor}
\node[anchor=west] at (19.465000\du,15.835000\du){$S_p$};
\definecolor{dialinecolor}{rgb}{0.000000, 0.000000, 0.000000}
\pgfsetstrokecolor{dialinecolor}
\node[anchor=west] at (4.3000\du,4.585000\du){$O1^*$ or $O2^*$};
\definecolor{dialinecolor}{rgb}{0.000000, 0.000000, 0.000000}
\pgfsetstrokecolor{dialinecolor}
\node[anchor=west] at (11.000\du,7.000\du){$O_1^*(\lambda_s^{(1)}, \beta^{(1)}) $};
\pgfsetlinewidth{0.050000\du}
\pgfsetdash{}{0pt}
\pgfsetdash{}{0pt}
\pgfsetbuttcap
{
\definecolor{dialinecolor}{rgb}{0.000000, 0.000000, 0.000000}
\pgfsetfillcolor{dialinecolor}
\pgfsetarrowsend{to}
\definecolor{dialinecolor}{rgb}{0.000000, 0.000000, 0.000000}
\pgfsetstrokecolor{dialinecolor}
\draw (12.465000\du,7.685000\du)--(11.715000\du,9.035000\du);
}
\definecolor{dialinecolor}{rgb}{0.000000, 0.000000, 0.000000}
\pgfsetstrokecolor{dialinecolor}
\node[anchor=west] at (6.4000\du,9.35000\du){$O_1^*(\lambda_s^{(2)},0) $};
\pgfsetlinewidth{0.050000\du}
\pgfsetdash{}{0pt}
\pgfsetdash{}{0pt}
\pgfsetbuttcap
{
\definecolor{dialinecolor}{rgb}{0.000000, 0.000000, 0.000000}
\pgfsetfillcolor{dialinecolor}
\pgfsetarrowsend{to}
\definecolor{dialinecolor}{rgb}{0.000000, 0.000000, 0.000000}
\pgfsetstrokecolor{dialinecolor}
\draw (7.765000\du,9.935000\du)--(8.715000\du,11.485000\du);
}
\definecolor{dialinecolor}{rgb}{0.000000, 0.000000, 0.000000}
\pgfsetstrokecolor{dialinecolor}
\node[anchor=west] at (9.000\du,17.685000\du){$I^-$};
\definecolor{dialinecolor}{rgb}{0.000000, 0.000000, 0.000000}
\pgfsetstrokecolor{dialinecolor}
\node[anchor=west] at (14.065000\du,17.685000\du){$I$};
\pgfsetlinewidth{0.050000\du}
\pgfsetdash{}{0pt}
\pgfsetdash{}{0pt}
\pgfsetbuttcap
{
\definecolor{dialinecolor}{rgb}{0.000000, 0.000000, 0.000000}
\pgfsetfillcolor{dialinecolor}
\pgfsetarrowsend{to}
\definecolor{dialinecolor}{rgb}{0.000000, 0.000000, 0.000000}
\pgfsetstrokecolor{dialinecolor}
\draw (12.515000\du,11.785000\du)--(11.465000\du,10.685000\du);
}
\definecolor{dialinecolor}{rgb}{0.000000, 0.000000, 0.000000}
\pgfsetstrokecolor{dialinecolor}
\node[anchor=west] at (11.3000\du,12.585000\du){$O_2^*(\lambda_s^{(4)}, \beta= \infty) $};
\end{tikzpicture}}
\caption{Contradiction from the concavity of $O^*_2$ {as derived in Theorem 1}}\label{fig:non_concave}
\end{figure}
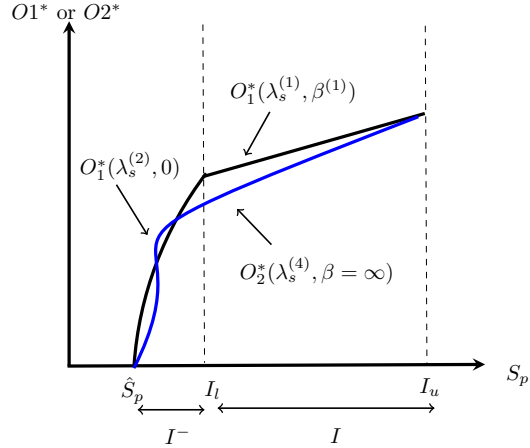 
 Mathematically, {{}the} conjecture states that $O_1^* > O_2^*$ for $S_p \in I^-$. Note that this comparison is not straight forward due to the non-linear dependence of the two objective functions\footnote{{One can possibly push the arguments for optimal objective of problem P1 being more than that of P2 in a small range around $I_l$ within $I^-$. Note that optimal objective of problem P1 is strictly more than that of P2 at ${S}_p = I_l$ (see \cite{Sudhir_standard_style}). Thus, by continuity, the optimal objective of problem P2 continues to be strictly lesser than that of P1 for some small range $\epsilon>0$ of $S_p$ when $S_p \in (I_l -\epsilon, I_l]$. But this leaves the possibility of {{}the} optimal objective of P2 being more than that of P1 beyond this small range, i.e., for $S_p \in (\hat{S}_p, I_l-\epsilon]$.  Thus, in the process, optimal objective of problem P2 will then be non-concave as in {Figure \ref{fig:non_concave}}. Our concavity arguments (in Theorem 1) preclude such a possibility (see Theorem 2).}}, $O_1^*$ in Equation (\ref{objective1}) and $O_2^*$ in Equation (\ref{eqn:objective_P2}) on $\lambda_s$ and $\beta$ which are in turn impacted by the value of $S_p$ in the interval of interest $I^-$. We present a complete proof for this conjecture and discuss its implications in subsequent section.

}

\section{A Proof of Conjecture and finite step algorithm} 
\label{con_proof}
{First, we prove few claims which are useful in obtaining the geometric structure (concavity) of $O_2^*$. Such geometric structure is stated and proved in Theorem 1. The conjecture is proved by contradiction in Theorem 2 by exploiting Theorem 1.  

}
\begin{clm}$\lambda_s^{(3)} > \lambda_s^{(1)}$, where $\lambda_s^{(3)}$ and $\lambda_s^{(1)}$ are the unique roots of cubics $\tilde{G}(\lambda_s)$ and $G(\lambda_s)$ respectively in the interval $(0,\mu-\lambda_p)$ with $G(\lambda_s)$ as:
\begin{equation}
2\mu\lambda_s^3 -[c\psi+\mu(a+4\phi_0)]\lambda_s^2 +2\phi_0[c\psi + \mu(a+\phi_0)]\lambda_s-a\mu\phi_0^2+c\psi\lambda_p(\mu+\phi_0)
\end{equation}
\begin{equation}\label{eqn:gtls}
\tilde{G}(\lambda_s) = 2\mu\lambda_s^3 -[a\mu+c\psi+4\mu^2]\lambda_s^2 +2\mu[a\mu+c\psi+\mu^2]\lambda_s-\mu[a\mu^2-c\psi\lambda_p]
\end{equation}
and $\phi_0 = \mu-\lambda_p$. 
\end{clm}
\begin{proof}
$\lambda_s^{(1)}$ is the unique root of cubic $G(\lambda_s)$ in the interval $(0, \mu-\lambda_p)$ whenever $\dfrac{a}{c} > \dfrac{\lambda_p(2\mu-\lambda_p)}{\mu(\mu-\lambda_p)^2}\psi$ \cite[Theorem 1]{Sudhir_standard_style}. Hence, $\lambda_s^{(1)} \in (0,\mu-\lambda_p) \text{ for } a \in (a_l, \infty)$ where $a_l =\dfrac{\lambda_p(2\mu-\lambda_p)}{\mu(\mu-\lambda_p)^2}c\psi $. $\lambda_s^{(3)}$ is the unique root of cubic $\tilde{G}(\lambda_s)$ in the interval $(0, \mu-\lambda_p)$  whenever $\dfrac{\mu-\lambda_p}{\mu\lambda_p} > \dfrac{a\lambda_p - c\psi}{2\mu\lambda_p^2 + c\psi(\mu+\lambda_p)}$ and $\dfrac{a}{c} > \dfrac{\lambda_p}{\mu^2}\psi$ \cite[Theorem 3]{Sudhir_standard_style}. Hence, $\lambda_s^{(3)} \in (0,\mu-\lambda_p)$ for $a \in (\tilde{a}_l, \tilde{a}_u)$ where $\tilde{a}_l = \dfrac{\lambda_p}{\mu^2}c\psi$ and $\tilde{a}_u = 2(\mu-\lambda_p) + \dfrac{c\psi}{\lambda_p}\left[1+\dfrac{\mu^2-\lambda_p^2}{\mu\lambda_p}\right]$. If $\dfrac{\mu-\lambda_p}{\mu\lambda_p} \leq \dfrac{a\lambda_p - c\psi}{2\mu\lambda_p^2 + c\psi(\mu+\lambda_p)}$, i.e., $a \ge \tilde{a}_u$ then $\tilde{G}(\mu-\lambda_p) \leq 0$ and $\tilde{G}(0) < 0$ hold by the definition of cubic $\tilde{G}(.)$ and $a>\tilde{a}_l$. Hence, it follows from Claim 3 in \cite{Sudhir_TR} that $\mu-\lambda_p \leq \lambda_s^{(3)} < \mu$. It follows that $\tilde{a}_l = \dfrac{\lambda_p}{\mu^2}c\psi < \dfrac{\lambda_p(2\mu-\lambda_p)}{\mu(\mu-\lambda_p)^2}c\psi = a_l$. Note that $G(0),~ \tilde{G}(0)$ and $\tilde{G}(\mu-\lambda_p)$ is 0 at $a=a_l, ~\tilde{a}_l \text{ and } \tilde{a}_u$ respectively. Hence, $\lambda_s^{(1)} = 0$, $\lambda_s^{(3)} = 0$ and $\lambda_s^{(3)} = \mu-\lambda_p$ at $a=a_l, ~\tilde{a}_l \text{ and } \tilde{a}_u$ respectively. Note that $a<\tilde{a}_l$ results in infeasible solution (See \cite[Theorem 1-4]{Sudhir_standard_style}). 

\begin{figure}[htbp]
  \centering
  \begin{minipage}[b]{0.4\textwidth}
\includegraphics[scale=0.22]{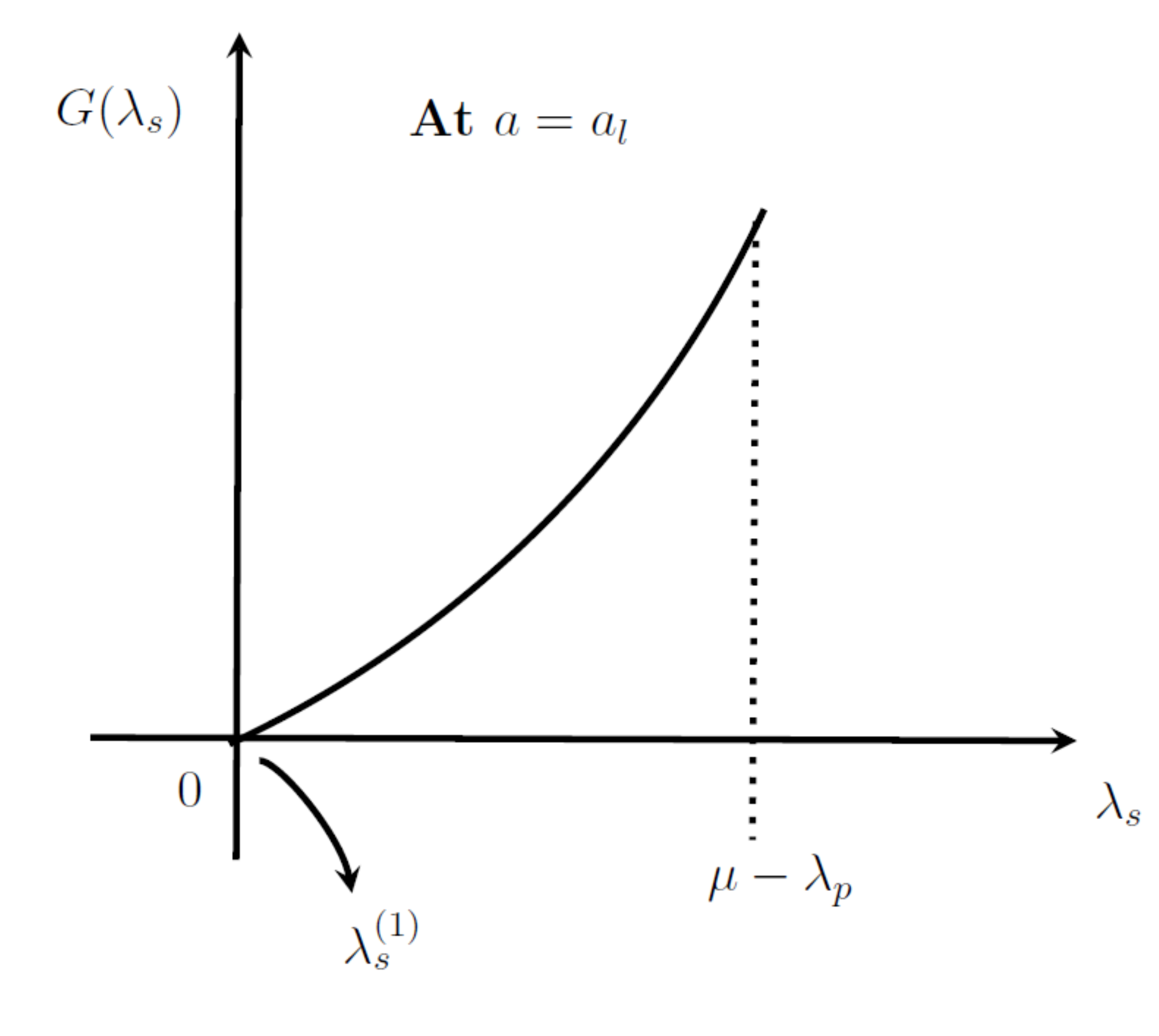}
\caption{$G(\lambda_s)$ vs $\lambda_s$ in range $(0, \mu-\lambda_p)$ at $a=a_l$}\label{fig:Gls}
  \end{minipage}\hfill
  \begin{minipage}[b]{0.4\textwidth}
	\includegraphics[scale=0.22]{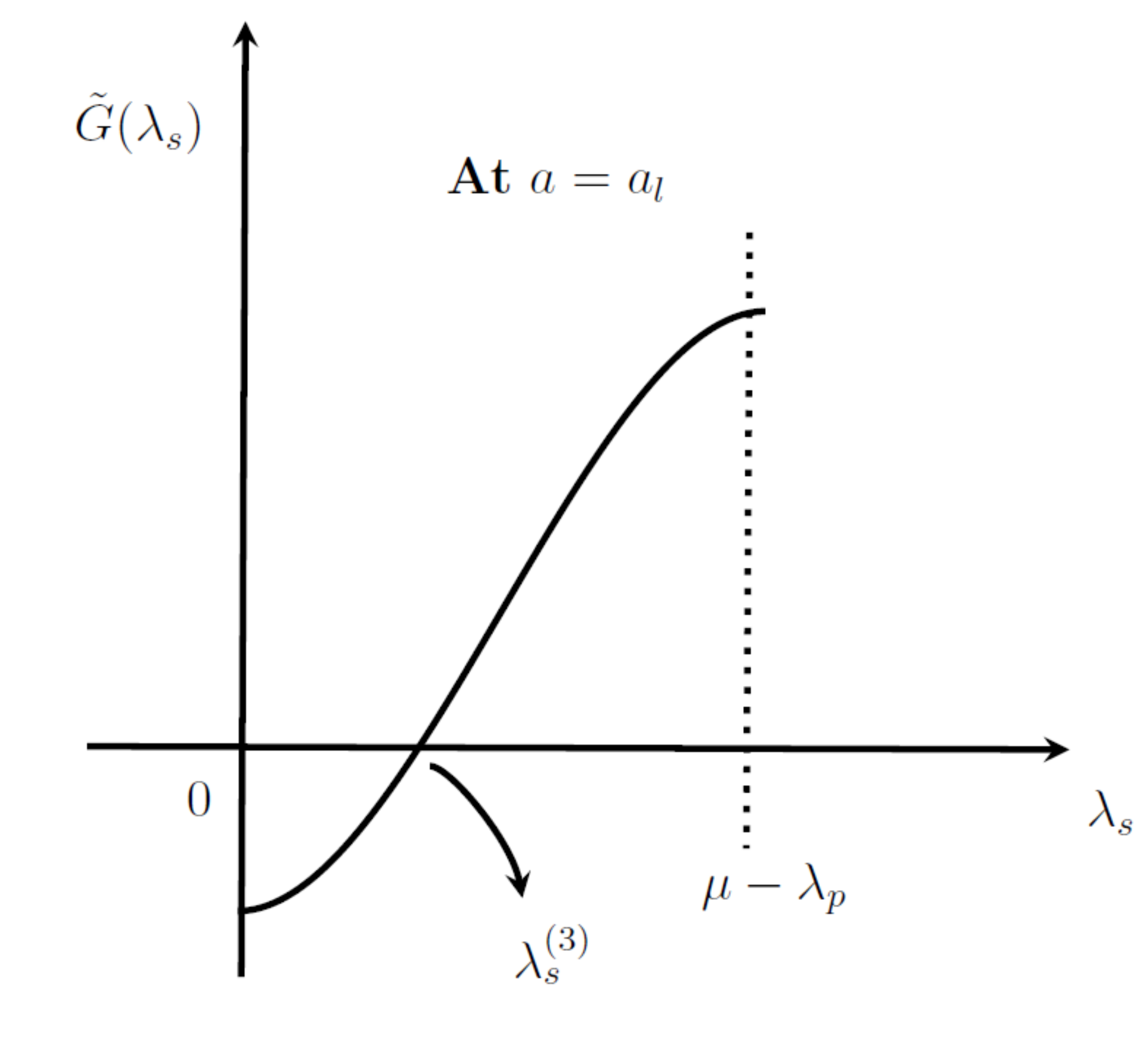}
 \caption{$\tilde{G}(\lambda_s)$ vs $\lambda_s$ in range $(0, \mu-\lambda_p)$ at $a=a_l$}\label{fig:Gtildals}
  \end{minipage}
\end{figure}

On the basis of relative values for $a_l,~\tilde{a}_l$ and $ \tilde{a}_u$ and using the fact that $\lambda_s^{(1)}$ and $\lambda_s^{(3)}$ are increasing functions of $a$ \cite[claim 5, page 24]{Sudhir_TR}, we have:
\begin{itemize}
\item If $a_l < \tilde{a}_u$, then,
\begin{enumerate}
\item $\lambda_s^{(1)} \leq 0$ and $0 < \lambda_s^{(3)} < \mu-\lambda_p$ for $a \in (\tilde{a}_l, a_l]$
\item $0<\lambda_s^{(1)} < \mu-\lambda_p$ and $0 < \lambda_s^{(3)} < \mu-\lambda_p$ for $a \in (a_l, \tilde{a}_u)$
\item $0<\lambda_s^{(1)} < \mu-\lambda_p$ and $\mu-\lambda_p \leq \lambda_s^{(3)} < \mu$ for $a \ge \tilde{a}_u$
\end{enumerate}
\item If $a_l \ge \tilde{a}_u$, then,
\begin{enumerate}
\item $\lambda_s^{(1)} < 0$ and $0 < \lambda_s^{(3)} < \mu-\lambda_p$ for $a \in (\tilde{a}_l, \tilde{a}_u)$
\item $\lambda_s^{(1)} \le 0$ and $\mu-\lambda_p \le \lambda_s^{(3)} < \mu$ for $a \in [\tilde{a}_u,a_l ]$
\item $0<\lambda_s^{(1)} < \mu-\lambda_p$ and $\mu-\lambda_p < \lambda_s^{(3)} < \mu$ for $a >a_l$
\end{enumerate}
\end{itemize}

 Hence, $\lambda_s^{(3)} > \lambda_s^{(1)}$ for all cases except the case when $a_l < \tilde{a}_u$ and $a \in (a_l, \tilde{a}_u)$. Note that $G(\lambda_s)$ and $\tilde{G}(\lambda_s)$ have exactly one root in interval $(0, \mu-\lambda_p)$. Root of $G(\lambda_s)$, $\lambda_s^{(1)}$, is 0 at $a=a_l$ and $G(\lambda_s)$ increases in the interval $(0, \mu-\lambda_p)$. Root of $\tilde{G}(\lambda_s)$, $\lambda_s^{(3)}$, is zero at $a=\tilde{a}_l < a_l$ and $\lambda_s^{(3)}$ is an increasing function of $a$. So $\lambda_s^{(3)} > \lambda_s^{(1)}$ at $a=a_l$. Plots of $G(\lambda_s)$ and $\tilde{G}(\lambda_s)$ are shown in Figure \ref{fig:Gls} and \ref{fig:Gtildals} respectively.

As $\dfrac{\partial G(\lambda_s)}{\partial a} = -\mu(\mu-\lambda_p - \lambda_s)^2 \text{ and }\dfrac{\partial \tilde{G}(\lambda_s)}{\partial a} =  -\mu(\mu-\lambda_s)^2$, $G(\lambda_s)$ and $\tilde{G}(\lambda_s)$ are decreasing functions of $a$. Since $\tilde{G}(\lambda_s)$ decreases with higher rate than $G(\lambda_s)$, $\lambda_s^{(3)}$ will increase with higher rate than $\lambda_s^{(1)}$. Hence, $\lambda_s^{(3)}>\lambda_s^{(1)}$ holds for $a \in (a_l, \tilde{a}_u)$ also and the claim follows. 
\end{proof}
The optimal solution of problem P1 is described in terms of primary class service level range $I$ and $I^-$ (see Theorem 1 and 2 in \cite{Sudhir_standard_style}), while the optimal solution of problem P2 is given by Theorem 3 and 4 of \cite{Sudhir_standard_style} in terms of service level range $J$ and $J^-$. {{}
From Theorem 3 and 4 of \cite{Sudhir_standard_style}, it is clear that if $\frac{\mu -\lambda_p}{\mu \lambda_p} \leq\frac{a\lambda_p-c\psi}{2\mu \lambda_p^2+c\psi(\mu +\lambda_p)}$ then $J^- = (\hat{S}_p,\infty)$ and $J=\phi$ otherwise $J^- = (\hat{S}_p,J_l]$ and $J = (J_l,\infty)$ where $J_l = \frac{\psi \lambda_3}{(\mu-\lambda_s^{(3)})(\mu-\lambda_3)}$ (see Figure \ref{intervals} for graphical illustration).} We present the following claim that relates $I$ and $I^-$ with the range $J$ and $J^-$. Claim 1 is exploited in proving following Claim 2. 
\begin{figure}[ht]
\centering
\resizebox{.5\textwidth}{!}{
\ifx\du\undefined
  \newlength{\du}
\fi
\setlength{\du}{15\unitlength}
\begin{tikzpicture}
\pgftransformxscale{1.000000}
\pgftransformyscale{-1.000000}
\definecolor{dialinecolor}{rgb}{0.000000, 0.000000, 0.000000}
\pgfsetstrokecolor{dialinecolor}
\definecolor{dialinecolor}{rgb}{1.000000, 1.000000, 1.000000}
\pgfsetfillcolor{dialinecolor}
\pgfsetlinewidth{0.100000\du}
\pgfsetdash{}{0pt}
\pgfsetdash{}{0pt}
\pgfsetbuttcap
{
\definecolor{dialinecolor}{rgb}{0.000000, 0.000000, 0.000000}
\pgfsetfillcolor{dialinecolor}
\definecolor{dialinecolor}{rgb}{0.000000, 0.000000, 0.000000}
\pgfsetstrokecolor{dialinecolor}
\draw[, , ->]  (10.000000\du,5.076393\du)--(29.800000\du,5.022714\du);
}
\pgfsetlinewidth{0.100000\du}
\pgfsetdash{}{0pt}
\pgfsetdash{}{0pt}
\pgfsetbuttcap
{
\definecolor{dialinecolor}{rgb}{0.000000, 0.000000, 0.000000}
\pgfsetfillcolor{dialinecolor}
\definecolor{dialinecolor}{rgb}{0.000000, 0.000000, 0.000000}
\pgfsetstrokecolor{dialinecolor}
\draw (10.000000\du,4.172714\du)--(9.950000\du,5.972714\du);
}
\pgfsetlinewidth{0.100000\du}
\pgfsetdash{}{0pt}
\pgfsetdash{}{0pt}
\pgfsetbuttcap
{
\definecolor{dialinecolor}{rgb}{0.000000, 0.000000, 0.000000}
\pgfsetfillcolor{dialinecolor}
\definecolor{dialinecolor}{rgb}{0.000000, 0.000000, 0.000000}
\pgfsetstrokecolor{dialinecolor}
\draw (13.066369\du,4.214083\du)--(13.016369\du,6.014083\du);
}
\pgfsetlinewidth{0.100000\du}
\pgfsetdash{}{0pt}
\pgfsetdash{}{0pt}
\pgfsetbuttcap
{
\definecolor{dialinecolor}{rgb}{0.000000, 0.000000, 0.000000}
\pgfsetfillcolor{dialinecolor}
\definecolor{dialinecolor}{rgb}{0.000000, 0.000000, 0.000000}
\pgfsetstrokecolor{dialinecolor}
\draw (16.031369\du,4.154083\du)--(15.981369\du,5.954083\du);
}

\pgfsetlinewidth{0.100000\du}
\pgfsetdash{}{0pt}
\pgfsetdash{}{0pt}
\pgfsetbuttcap
{
\definecolor{dialinecolor}{rgb}{0.000000, 0.000000, 0.000000}
\pgfsetfillcolor{dialinecolor}
\definecolor{dialinecolor}{rgb}{0.000000, 0.000000, 0.000000}
\pgfsetstrokecolor{dialinecolor}
\draw (18.34\du,4.544083\du)--(18.3496369\du,5.544083\du);
}
\pgfsetlinewidth{0.100000\du}
\pgfsetdash{}{0pt}
\pgfsetdash{}{0pt}
\pgfsetbuttcap
{
\definecolor{dialinecolor}{rgb}{0.000000, 0.000000, 0.000000}
\pgfsetfillcolor{dialinecolor}
\definecolor{dialinecolor}{rgb}{0.000000, 0.000000, 0.000000}
\pgfsetstrokecolor{dialinecolor}
\draw (20.046369\du,4.144083\du)--(19.996369\du,5.944083\du);
}
\pgfsetlinewidth{0.100000\du}
\pgfsetdash{}{0pt}
\pgfsetdash{}{0pt}
\pgfsetbuttcap
{
\definecolor{dialinecolor}{rgb}{0.000000, 0.000000, 0.000000}
\pgfsetfillcolor{dialinecolor}
\definecolor{dialinecolor}{rgb}{0.000000, 0.000000, 0.000000}
\pgfsetstrokecolor{dialinecolor}
\draw (22.011369\du,4.184083\du)--(21.961369\du,5.984083\du);
}
\pgfsetlinewidth{0.100000\du}
\pgfsetdash{}{0pt}
\pgfsetdash{}{0pt}
\pgfsetbuttcap
{
\definecolor{dialinecolor}{rgb}{0.000000, 0.000000, 0.000000}
\pgfsetfillcolor{dialinecolor}
\definecolor{dialinecolor}{rgb}{0.000000, 0.000000, 0.000000}
\pgfsetstrokecolor{dialinecolor}
\draw (27.150000\du,4.172714\du)--(25.900000\du,6.172714\du);
}
\pgfsetlinewidth{0.100000\du}
\pgfsetdash{}{0pt}
\pgfsetdash{}{0pt}
\pgfsetbuttcap
{
\definecolor{dialinecolor}{rgb}{0.000000, 0.000000, 0.000000}
\pgfsetfillcolor{dialinecolor}
\definecolor{dialinecolor}{rgb}{0.000000, 0.000000, 0.000000}
\pgfsetstrokecolor{dialinecolor}
\draw (27.800000\du,4.172714\du)--(26.583131\du,6.180845\du);
}
\definecolor{dialinecolor}{rgb}{0.000000, 0.000000, 0.000000}
\pgfsetstrokecolor{dialinecolor}
\node[anchor=west] at (9.40000\du,6.772714\du){$0$};
\definecolor{dialinecolor}{rgb}{0.000000, 0.000000, 0.000000}
\pgfsetstrokecolor{dialinecolor}
\node[anchor=west] at (12.40000\du,6.822714\du){$\hat{S}_p$};
\definecolor{dialinecolor}{rgb}{0.000000, 0.000000, 0.000000}
\pgfsetstrokecolor{dialinecolor}
\node[anchor=west] at (15.550000\du,6.722714\du){$I_l$};

\definecolor{dialinecolor}{rgb}{0.000000, 0.000000, 0.000000}
\pgfsetstrokecolor{dialinecolor}
\node[anchor=west] at (17.500000\du,6.0\du){$I_{FCFS}$};
\definecolor{dialinecolor}{rgb}{0.000000, 0.000000, 0.000000}
\pgfsetstrokecolor{dialinecolor}
\node[anchor=west] at (19.500000\du,6.672714\du){$I_u$};
\definecolor{dialinecolor}{rgb}{0.000000, 0.000000, 0.000000}
\pgfsetstrokecolor{dialinecolor}
\node[anchor=west] at (21.400000\du,6.572714\du){$J_l$};
\definecolor{dialinecolor}{rgb}{0.000000, 0.000000, 0.000000}
\pgfsetstrokecolor{dialinecolor}
\node[anchor=west] at (29.5100000\du,5.96102714\du){$S_p$};
\pgfsetlinewidth{0.100000\du}
\pgfsetdash{}{0pt}
\pgfsetdash{}{0pt}
\pgfsetbuttcap
\definecolor{dialinecolor}{rgb}{0.000000, 0.000000, 0.000000}
\pgfsetstrokecolor{dialinecolor}
\node[anchor=west] at (25.100000\du,6.722714\du){$\infty$};
\pgfsetlinewidth{0.100000\du}
\pgfsetdash{}{0pt}
\pgfsetdash{}{0pt}
\pgfsetbuttcap
{
\definecolor{dialinecolor}{rgb}{0.000000, 0.000000, 0.000000}
\pgfsetfillcolor{dialinecolor}
\pgfsetarrowsstart{stealth}
\pgfsetarrowsend{stealth}
\definecolor{dialinecolor}{rgb}{0.000000, 0.000000, 0.000000}
\pgfsetstrokecolor{dialinecolor}
\draw (13.100000\du,3.872714\du)--(15.850000\du,3.872714\du);
}
\pgfsetlinewidth{0.100000\du}
\pgfsetdash{}{0pt}
\pgfsetdash{}{0pt}
\pgfsetbuttcap
{
\definecolor{dialinecolor}{rgb}{0.000000, 0.000000, 0.000000}
\pgfsetfillcolor{dialinecolor}
\pgfsetarrowsstart{stealth}
\pgfsetarrowsend{stealth}
\definecolor{dialinecolor}{rgb}{0.000000, 0.000000, 0.000000}
\pgfsetstrokecolor{dialinecolor}
\draw (16.300000\du,3.922714\du)--(19.800000\du,3.922714\du);
}
\definecolor{dialinecolor}{rgb}{0.000000, 0.000000, 0.000000}
\pgfsetstrokecolor{dialinecolor}
\node[anchor=west] at (14.00000\du,3.322714\du){$I^-$};
\definecolor{dialinecolor}{rgb}{0.000000, 0.000000, 0.000000}
\pgfsetstrokecolor{dialinecolor}
\node[anchor=west] at (17.550000\du,3.422714\du){$I$};
\pgfsetlinewidth{0.100000\du}
\pgfsetdash{}{0pt}
\pgfsetdash{}{0pt}
\pgfsetbuttcap
{
\definecolor{dialinecolor}{rgb}{0.000000, 0.000000, 0.000000}
\pgfsetfillcolor{dialinecolor}
\pgfsetarrowsstart{stealth}
\pgfsetarrowsend{stealth}
\definecolor{dialinecolor}{rgb}{0.000000, 0.000000, 0.000000}
\pgfsetstrokecolor{dialinecolor}
\draw (13.100000\du,2.122714\du)--(29.800000\du,2.122714\du);
}
\pgfsetlinewidth{0.100000\du}
\pgfsetdash{}{0pt}
\pgfsetdash{}{0pt}
\pgfsetbuttcap
{
\definecolor{dialinecolor}{rgb}{0.000000, 0.000000, 0.000000}
\pgfsetfillcolor{dialinecolor}
\pgfsetarrowsstart{stealth}
\pgfsetarrowsend{stealth}
\definecolor{dialinecolor}{rgb}{0.000000, 0.000000, 0.000000}
\pgfsetstrokecolor{dialinecolor}
\draw (12.800000\du,7.922714\du)--(21.800000\du,7.922714\du);
}
\pgfsetlinewidth{0.100000\du}
\pgfsetdash{}{0pt}
\pgfsetdash{}{0pt}
\pgfsetbuttcap
{
\definecolor{dialinecolor}{rgb}{0.000000, 0.000000, 0.000000}
\pgfsetfillcolor{dialinecolor}
\pgfsetarrowsstart{stealth}
\pgfsetarrowsend{stealth}
\definecolor{dialinecolor}{rgb}{0.000000, 0.000000, 0.000000}
\pgfsetstrokecolor{dialinecolor}
\draw (22.150000\du,7.972714\du)--(30.000000\du,7.972714\du);
}
\definecolor{dialinecolor}{rgb}{0.000000, 0.000000, 0.000000}
\pgfsetstrokecolor{dialinecolor}
\node[anchor=west] at (17.0000\du,8.60022714\du){$J^-$};
\definecolor{dialinecolor}{rgb}{0.000000, 0.000000, 0.000000}
\pgfsetstrokecolor{dialinecolor}
\node[anchor=west] at (25.150000\du,8.602714\du){$J$};
\definecolor{dialinecolor}{rgb}{0.000000, 0.000000, 0.000000}
\pgfsetstrokecolor{dialinecolor}
\node[anchor=west] at (19.600000\du,1.0372714\du){$J^- $};
\definecolor{dialinecolor}{rgb}{0.000000, 0.000000, 0.000000}
\pgfsetstrokecolor{dialinecolor}
\node[anchor=west] at (20.900000\du,1.0372714\du){If $\dfrac{\mu -\lambda_p}{\mu \lambda_p} \leq\dfrac{a\lambda_p-c\psi}{2\mu \lambda_p^2+c\psi(\mu +\lambda_p)}$};
\definecolor{dialinecolor}{rgb}{0.000000, 0.000000, 0.000000}
\pgfsetstrokecolor{dialinecolor}
\node[anchor=west] at (17.200000\du,10.272714\du){If $\dfrac{\mu -\lambda_p}{\mu \lambda_p} > \dfrac{a\lambda_p-c\psi}{2\mu \lambda_p^2+c\psi(\mu +\lambda_p)}$};
\end{tikzpicture}}
\caption{Relation among intervals of $S_p$}\label{intervals}
\end{figure}
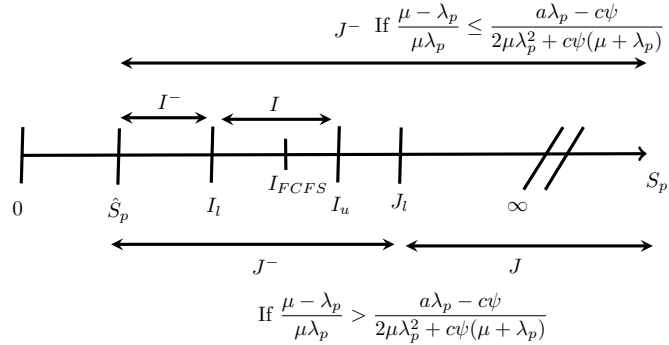
\begin{clm} \textit{ Range $I^-\cup I$ is contained in $J^-$, i.e., $I^- \cup I \subset J^-$.}\end{clm}
\begin{proof}

The solution of optimization problem P2 is given by Theorem 3 and 4 of \cite{Sudhir_standard_style} in terms of the service level range $J$ and $J^-$ (see Figure \ref{intervals}). So, the entire feasible range of service level $(\hat{S}_p,\infty)$ is divided in interval $J^-\cup J$ as shown in Figure \ref{intervals}. 

%

From Theorem 3 and 4 of \cite{Sudhir_standard_style}, it is clear that if $\dfrac{\mu -\lambda_p}{\mu \lambda_p} \leq\dfrac{a\lambda_p-c\psi}{2\mu \lambda_p^2+c\psi(\mu +\lambda_p)}$ then $J^- = (\hat{S}_p,\infty)$ and $J=\phi$ otherwise $J^- = (\hat{S}_p,J_l]$ and $J = (J_l,\infty)$ where $J_l = \dfrac{\psi \lambda_3}{(\mu-\lambda_s^{(3)})(\mu-\lambda_3)}$ and $\hat{S}_p = \dfrac{\psi\lambda_p}{\mu(\mu-\lambda_p)}$. Now, consider the following two cases to prove the claim:

\indent \textit{\textbf{Case 1:}} When $\dfrac{\mu -\lambda_p}{\mu \lambda_p} \leq \dfrac{a\lambda_p-c\psi}{2\mu \lambda_p^2+c\psi(\mu +\lambda_p)}$\\
Interval $J^-$ becomes $(\hat{S}_p,\infty)$ under the condition of this case as discussed above. Since lower and upper limit of $I^-\cup I$ are $\hat{S}_p$ and $I_u$ respectively, these limits are finite (see \cite{Sudhir_standard_style}). Hence, $I^- \cup I \subset J^-$ holds.\\

\indent \textit{\textbf{Case 2:}} When $\dfrac{\mu -\lambda_p}{\mu \lambda_p} > \dfrac{a\lambda_p-c\psi}{2\mu \lambda_p^2+c\psi(\mu +\lambda_p)}$

In this case, $J^- = (\hat{S}_p,J_l]$ where $J_l = \dfrac{\psi\lambda_3}{(\mu-\lambda_s^{(3)})(\mu-\lambda_3)}$ and $I^-\cup I = (\hat{S}_p,I_u)$ where $I_u = $ $\frac{\psi\lambda_1}{(\mu-\lambda_s^{(1)})(\mu-\lambda_1)} $ and by definition $\lambda_1 = \lambda_p + \lambda_s^{(1)},~ \lambda_3 = \lambda_p +\lambda_s^{(3)}$. Note that, $I_u = \xi(\lambda_s^{(1)})$ and $J_l = \xi(\lambda_s^{(3)})$ where $\xi(\lambda_s) = \dfrac{\psi\lambda}{(\mu-\lambda_s)(\mu-\lambda)}$ and $\lambda = \lambda_p + \lambda_s$. On computing partial derivative of $\xi(\lambda_s)$ with respect to $\lambda_s$, we get $\dfrac{\partial\xi(\lambda_s)}{\partial\lambda_s}=\dfrac{\psi(\mu(\mu-\lambda_s)+\lambda(\mu - \lambda))}{(\mu - \lambda_s)^2(\mu - \lambda)^2} >0$, i.e., $\xi(\lambda_s)$ is an increasing function of $\lambda_s$. So, $\xi(\lambda_s^{(3)})>\xi(\lambda_s^{(1)})$ iff $\lambda_s^{(3)}>\lambda_s^{(1)}$, But $\lambda_s^{(3)} > \lambda_s^{(1)}$ follows from claim 1. Thus, $\xi(\lambda_s^{(3)})>\xi(\lambda_s^{(1)})$ follows  and equivalently $J_l > I_u$ holds. Hence $I^- \cup I \subset J^-$ holds in this case too, and the claim follows.
\end{proof}

It follows from Claim 2 that the nature of the optimal objective of problem P2 in service level range $I^- \cup I$ can be identified by finding the same in service level range $J^-$. Such nature {{}(concave increasing)} of the objective function is identified in the following theorem and will be useful in proving the conjecture. 

\begin{thm}
The optimal objective function for problem $P2$, i.e., $O_2^*$ is increasing concave in service level range $I^- \cup I$, while the optimal objective function for problem P1, $O_1^*$, is increasing concave in $I^-$ and linearly increasing in $I$.
\end{thm}
\begin{proof}

It is shown in \cite{Sudhir_standard_style} that $O_1^*$ is increasing concave in $I^-$ and linearly increasing in $I$, we give details here for its completeness. The optimal objective of problem P1 and P2 are given by $O_{1}^{*}$ and $O_{2}^{*}$, and the corresponding optimal solutions are given by $(\lambda_s^f,\beta^f)$ and $(\lambda_s^i,\infty)$, respectively. Superscript `$f$' and `$i$' denote solution corresponding to finite and infinite $\beta$. In case of finite $\beta$, solution is given by Theorem 1 and 2 in \cite{Sudhir_standard_style}, for $S_p \in I^-\cup I$. Hence, the waiting time constraint is binding ($W_p(\lambda_s,\beta) \leq S_p$) from statement of these theorems. 

In case of infinite $\beta$, the solution is given by Theorem 4 in \cite{Sudhir_standard_style}, for $S_p \in I^- \cup I$ as $I^- \cup I \subset J^-$ (Claim 2). It follows from Theorem 4 in \cite{Sudhir_standard_style} that the primary class customer's waiting time constraint is  binding. Therefore, waiting time constraint $W_p(\lambda_s,\beta) \leq S_p$ is always binding for $S_p \in I^- \cup I$ irrespective of nature of optimal solution (finite or infinite $\beta$). By using the interpretation of Lagrange multiplier (Proposition 3.3.3 in \cite[page~315]{bertsekas} and \cite[page 25]{Sudhir_TR}), we have
\begin{equation}\label{eqn:22}
\dfrac{\partial O_1^*}{\partial S_p} = -u_1^f~\text{and}~\dfrac{\partial O_2^*}{\partial S_p} = -v_1^i 
\end{equation}
where $u_1^f$ and $v_1^i$ are the corresponding values of the Lagrangian multipliers associated with the constraint $W_p(\lambda_s,\beta) = S_p$ of the optimization problems P1 and P2 respectively. As defined in \cite[page~25]{Sudhir_TR}:
\begin{eqnarray}\label{eqn:23}
u_1^f = \dfrac{(\mu-\lambda_p)G(\lambda_s^f)}{b\psi(\mu-\lambda_p-\lambda_s^f)^2}-\dfrac{c\lambda_p}{b}~~\text{and}\\ v_1^i = \dfrac{(\mu-\lambda_p-\lambda_s^i)^2\tilde{G}(\lambda_s^i)}{b\psi\mu[\mu(\mu+\lambda_p)-(\lambda_p+\lambda_s^i)^2]}\label{eqn:vi}
\end{eqnarray}
Now, we look for the sign of $v_1^i$ and $u_1^f$ in service level range $I^- \cup I$ to study the nature of {{}the} optimal objective function.

The optimal objective of optimization problem P2, i.e., $O_2^*$ is given by Theorem 4 of \cite{Sudhir_standard_style} for service level range $S_p \in I^-\cup I$ as $ I^-\cup I \subset J^-$. $\lambda_s^{(4)}$ is the optimal admission rate for secondary class customers in Theorem 4. Hence, $\lambda_s^i = \lambda_s^{(4)}$. Sign of $v_1^i$ is decided by $\tilde{G}(\lambda_s^i)$. $\lambda_s^{(3)}$ is the root of cubic $\tilde{G}(\lambda_s)$ as discussed in Claim 1. It follows that $\lambda_s^{(4)} < \lambda_s^{(3)}$ (see proof of Theorem 4 in \cite{Sudhir_TR}). We note that $\tilde{G}(\lambda_s)$ is negative and increasing in interval $[0,\lambda_s^{(3)}]$. So,  $\tilde{G}(\lambda_s^i) = \tilde{G}(\lambda_s^{(4)})\leq \tilde{G}(\lambda_s^{(3)}) = 0 $ and hence $v_1^i\leq 0$ for $S_p \in I^-\cup I $. Now, it follows from Equation (\ref{eqn:22}) that $\dfrac{\partial O_2^*}{\partial S_p} \geq 0$.

The solution of optimization problem P1 is given by Theorem 1 of \cite{Sudhir_standard_style} for $S_p \in I$ with $\lambda_s^{(1)}$ as the optimal admission rate for secondary class customers. Hence, $\lambda_s^f = \lambda_s^{(1)}$. $\lambda_s^{(1)}$ is the root of cubic $G(\lambda_s)$. From Equation (\ref{eqn:23}), $u_1^f = \dfrac{-c\lambda_p}{b} \leq 0$ for $S_p \in I$. 

Solution of problem P1 is given by Theorem 2 of \cite{Sudhir_standard_style} for $S_p \in I^-$ with $\lambda_s^{(2)} = \dfrac{\mu(\mu-\lambda_p)S_p}{\psi}-\lambda_p$ as the optimal admission rate for secondary class customers. Hence, $\lambda_s^f = \lambda_s^{(2)}$. Note that $\lambda_s^{(2)}$ linearly increases with $S_p$ and $\lambda_s^{(2)} = \lambda_s^{(1)}$ at $S_p = \dfrac{\psi (\lambda_p + \lambda_s^{(1)})}{\mu(\mu-\lambda_p)} = I_l$, the upper limit of interval $I^-$. Thus, $\lambda_s^{(2)} < \lambda_s^{(1)}$ for $S_p \in I^- \equiv (\hat{S}_p, I_l)$. $G(\lambda_s)$ is an increasing function of $\lambda_s \in (0, \lambda_s^{(1)})$ (see proof of claim 1 in \cite{Sudhir_TR}). This implies that $G(\lambda_s^{(2)}) \le G(\lambda_s^{(1)}) = 0$. From Equation (\ref{eqn:23}), $u_1^f \le 0$ for $S_p \in I^-$.

 Thus $u_1^f \le 0$ for $S_p \in I^-\cup I$ and we get the following from Equation (\ref{eqn:22}): 
\begin{equation}\label{eqn:25}
\dfrac{\partial O_1^*}{\partial S_p} \geq 0~\text{and}~\dfrac{\partial O_2^*}{\partial S_p} \geq 0
\end{equation}
Thus, $O_1^*$ and $O_2^*$ are increasing functions of $S_p$, in the interval $I^-\cup I$. Partial derivatives of Lagrangian multipliers with respect to $\lambda_s^f$ and $\lambda_s^i$ are shown to be positive \cite[page 25]{Sudhir_TR}:
\begin{equation} \label{eqn:dudvls}
\dfrac{\partial u_1^f}{\partial \lambda_s^f} \geq 0~\text{and}~\dfrac{\partial v_1^i}{\partial \lambda_s^i} \geq 0\\
\end{equation}
From Equation (\ref{eqn:22}):
\begin{equation}\label{eqn:27}
\dfrac{\partial^2 O_1^*}{\partial S_p^2} = -\dfrac{\partial u_1^f}{\partial S_p} = -\dfrac{\partial u_1^f}{\partial \lambda_s^f}\dfrac{\partial \lambda_s^f}{\partial S_p}
\end{equation}
Corollary 1 in \cite{Sudhir_standard_style} states that the mean arrival rate of secondary class customers $\lambda_s^{(1)}$ is independent of $S_p$ in interval $I$, so, $\dfrac{\partial \lambda_s^{(1)}}{\partial S_p} = 0$. But for $S_p \in I$, $\lambda_s^f = \lambda_s^{(1)}$ and hence we have
\begin{equation}\label{eqn:28}
\dfrac{\partial^2 O_1^*}{\partial S_p^2} =0 
\end{equation}
Consider Corollary 2 in \cite{Sudhir_standard_style} which states that the mean arrival rate of secondary class customers $\lambda_s^{(2)}$ is linearly increasing function of $S_p$ in interval $I^-$, i.e., $\dfrac{\partial \lambda_s^{(2)}}{\partial S_p} > 0$. But for $S_p \in I^-$, $\lambda_s^f = \lambda_s^{(2)}$ and we get the following from Equation (\ref{eqn:27}),
\begin{equation}\label{eqn:29}
\dfrac{\partial^2 O_1^*}{\partial S_p^2} \leq 0 
\end{equation}

By using equations (\ref{eqn:25}), (\ref{eqn:28}) and (\ref{eqn:29}), we can say that $O_1^*$ is a linearly increasing function of $S_p$, in the interval $I$, while it is an increasing concave function of $S_p$ in the interval $I^-$. 
For $S_p \in I^- \cup I$ with $\beta$ as infinity, solution of problem P2 is given by Theorem 4 of \cite{Sudhir_standard_style} with $\lambda_s^{(4)}$ as the optimal admission rate for secondary class customers. Hence, $\lambda_s^i = \lambda_s^{(4)}$. From Equation (\ref{eqn:22}), we have
\begin{equation}\label{eqn:d2O2sp}
\dfrac{\partial^2 O_2^*}{\partial S_p^2} = -\dfrac{\partial v_1^i}{\partial S_p} = -\dfrac{\partial v_1^i}{\partial \lambda_s^i}\dfrac{\partial \lambda_s^i}{\partial S_p}
\end{equation}
Consider Corollary 3 in \cite{Sudhir_standard_style} which states that $\lambda_s^{(4)}$ is an increasing function of $S_p$ in the interval $J^-$. Since $I^-\cup I \subset J^-$,  $\dfrac{\partial \lambda_s^i}{\partial S_p} \geq 0$ for $S_p \in I^-\cup I $. So from equation (\ref{eqn:dudvls}) and (\ref{eqn:d2O2sp}), we have  
\begin{equation}
\dfrac{\partial^2 O_2^*}{\partial S_p^2} \leq 0~\text{for}~S_p\in I^-\cup I
\end{equation}
$O_2^*$ is, thus, an  increasing function of $S_p$ as $\dfrac{\partial O_2^*}{\partial S_p}\ge 0 $ and concave as $\dfrac{\partial^2 O_2^*}{\partial S_p^2}\le 0$ for $ S_p \in I^- \cup I$. Hence, the theorem follows.

\begin{figure}[htb]
  \centering
  \begin{minipage}[b]{0.4\textwidth}
\resizebox{1 \textwidth}{!}{
\ifx\du\undefined
  \newlength{\du}
\fi
\setlength{\du}{15\unitlength}
\begin{tikzpicture}
\pgftransformxscale{1.000000}
\pgftransformyscale{-1.000000}
\definecolor{dialinecolor}{rgb}{0.000000, 0.000000, 0.000000}
\pgfsetstrokecolor{dialinecolor}
\definecolor{dialinecolor}{rgb}{1.000000, 1.000000, 1.000000}
\pgfsetfillcolor{dialinecolor}
\pgfsetlinewidth{0.100000\du}
\pgfsetdash{}{0pt}
\pgfsetdash{}{0pt}
\pgfsetbuttcap
{
\definecolor{dialinecolor}{rgb}{0.000000, 0.000000, 0.000000}
\pgfsetfillcolor{dialinecolor}
\pgfsetarrowsend{stealth}
\definecolor{dialinecolor}{rgb}{0.000000, 0.000000, 0.000000}
\pgfsetstrokecolor{dialinecolor}
\draw (10.000000\du,12.150000\du)--(22.800000\du,12.100000\du);
}
\pgfsetlinewidth{0.100000\du}
\pgfsetdash{}{0pt}
\pgfsetdash{}{0pt}
\pgfsetbuttcap
{
\definecolor{dialinecolor}{rgb}{0.000000, 0.000000, 0.000000}
\pgfsetfillcolor{dialinecolor}
\pgfsetarrowsend{stealth}
\definecolor{dialinecolor}{rgb}{0.000000, 0.000000, 0.000000}
\pgfsetstrokecolor{dialinecolor}
\draw (10.050000\du,12.150000\du)--(10.050000\du,1.500000\du);
}
\pgfsetlinewidth{0.000000\du}
\pgfsetdash{{1.000000\du}{1.000000\du}}{0\du}
\pgfsetdash{{0.200000\du}{0.200000\du}}{0\du}
\pgfsetbuttcap
{
\definecolor{dialinecolor}{rgb}{0.000000, 0.000000, 0.000000}
\pgfsetfillcolor{dialinecolor}
\definecolor{dialinecolor}{rgb}{0.000000, 0.000000, 0.000000}
\pgfsetstrokecolor{dialinecolor}
\draw (14.150000\du,1.600000\du)--(14.200000\du,12.050000\du);
}
\pgfsetlinewidth{0.000000\du}
\pgfsetdash{{0.200000\du}{0.200000\du}}{0\du}
\pgfsetdash{{0.200000\du}{0.200000\du}}{0\du}
\pgfsetbuttcap
{
\definecolor{dialinecolor}{rgb}{0.000000, 0.000000, 0.000000}
\pgfsetfillcolor{dialinecolor}
\definecolor{dialinecolor}{rgb}{0.000000, 0.000000, 0.000000}
\pgfsetstrokecolor{dialinecolor}
\draw (20.965239\du,1.640239\du)--(21.015239\du,12.090239\du);
}
\pgfsetlinewidth{0.100000\du}
\pgfsetdash{}{0pt}
\pgfsetdash{}{0pt}
\pgfsetbuttcap
{
\definecolor{dialinecolor}{rgb}{0.000000, 0.000000, 0.000000}
\pgfsetfillcolor{dialinecolor}
\definecolor{dialinecolor}{rgb}{0.000000, 0.000000, 0.000000}
\pgfsetstrokecolor{dialinecolor}
\draw (14.150000\du,6.300000\du)--(20.950000\du,4.350000\du);
}
\pgfsetlinewidth{0.100000\du}
\pgfsetdash{}{0pt}
\pgfsetdash{}{0pt}
\pgfsetbuttcap
{
\definecolor{dialinecolor}{rgb}{0.000000, 0.000000, 0.000000}
\pgfsetfillcolor{dialinecolor}
\definecolor{dialinecolor}{rgb}{0.000000, 0.000000, 0.000000}
\pgfsetstrokecolor{dialinecolor}
\pgfpathmoveto{\pgfpoint{14.250279\du}{6.199628\du}}
\pgfpatharc{217}{184}{11.103314\du and 11.103314\du}
\pgfusepath{stroke}
}
\pgfsetlinewidth{0.100000\du}
\pgfsetdash{}{0pt}
\pgfsetdash{}{0pt}
\pgfsetmiterjoin
\pgfsetbuttcap
{
\definecolor{dialinecolor}{rgb}{0.000000, 0.000000, 0.000000}
\pgfsetfillcolor{dialinecolor}
\definecolor{dialinecolor}{rgb}{0.000000, 0.000000, 1.000000}
\pgfsetstrokecolor{dialinecolor}
\pgfpathmoveto{\pgfpoint{12.000000\du}{12.000000\du}}
\pgfpathcurveto{\pgfpoint{12.500000\du}{8.700000\du}}{\pgfpoint{17.800000\du}{5.050000\du}}{\pgfpoint{20.750000\du}{4.450000\du}}
\pgfusepath{stroke}
}
\pgfsetlinewidth{0.050000\du}
\pgfsetdash{}{0pt}
\pgfsetdash{}{0pt}
\pgfsetbuttcap
{
\definecolor{dialinecolor}{rgb}{0.000000, 0.000000, 0.000000}
\pgfsetfillcolor{dialinecolor}
\pgfsetarrowsstart{to}
\pgfsetarrowsend{to}
\definecolor{dialinecolor}{rgb}{0.000000, 0.000000, 0.000000}
\pgfsetstrokecolor{dialinecolor}
\draw (12.150000\du,13.500000\du)--(14.200000\du,13.500000\du);
}
\pgfsetlinewidth{0.050000\du}
\pgfsetdash{}{0pt}
\pgfsetdash{}{0pt}
\pgfsetbuttcap
{
\definecolor{dialinecolor}{rgb}{0.000000, 0.000000, 0.000000}
\pgfsetfillcolor{dialinecolor}
\pgfsetarrowsstart{to}
\pgfsetarrowsend{to}
\definecolor{dialinecolor}{rgb}{0.000000, 0.000000, 0.000000}
\pgfsetstrokecolor{dialinecolor}
\draw (14.600000\du,13.450000\du)--(21.250000\du,13.450000\du);
}
\definecolor{dialinecolor}{rgb}{0.000000, 0.000000, 0.000000}
\pgfsetstrokecolor{dialinecolor}
\node[anchor=west] at (11.4500000\du,12.900000\du){$\hat{S}_p$};
\definecolor{dialinecolor}{rgb}{0.000000, 0.000000, 0.000000}
\pgfsetstrokecolor{dialinecolor}
\node[anchor=west] at (13.80000\du,12.900000\du){$I_l$};
\definecolor{dialinecolor}{rgb}{0.000000, 0.000000, 0.000000}
\pgfsetstrokecolor{dialinecolor}
\node[anchor=west] at (20.700000\du,12.800000\du){$I_u$};
\definecolor{dialinecolor}{rgb}{0.000000, 0.000000, 0.000000}
\pgfsetstrokecolor{dialinecolor}
\node[anchor=west] at (23.200000\du,12.450000\du){$S_p$};
\definecolor{dialinecolor}{rgb}{0.000000, 0.000000, 0.000000}
\pgfsetstrokecolor{dialinecolor}
\node[anchor=west] at (8.34680000\du,1.200000\du){$O_1^*$ or $O_2^*$};
\definecolor{dialinecolor}{rgb}{0.000000, 0.000000, 0.000000}
\pgfsetstrokecolor{dialinecolor}
\node[anchor=west] at (14.46900000\du,3.820000\du){$O_1^*(\lambda_s^{(1)}, \beta^{(1)})$};
\pgfsetlinewidth{0.050000\du}
\pgfsetdash{}{0pt}
\pgfsetdash{}{0pt}
\pgfsetbuttcap
{
\definecolor{dialinecolor}{rgb}{0.000000, 0.000000, 0.000000}
\pgfsetfillcolor{dialinecolor}
\pgfsetarrowsend{to}
\definecolor{dialinecolor}{rgb}{0.000000, 0.000000, 0.000000}
\pgfsetstrokecolor{dialinecolor}
\draw (16.200000\du,4.300000\du)--(15.450000\du,5.650000\du);
}
\definecolor{dialinecolor}{rgb}{0.000000, 0.000000, 0.000000}
\pgfsetstrokecolor{dialinecolor}
\node[anchor=west] at (10.000000\du,6.00000\du){$O_1^*(\lambda_s^{(2)}, 0)$};
\pgfsetlinewidth{0.050000\du}
\pgfsetdash{}{0pt}
\pgfsetdash{}{0pt}
\pgfsetbuttcap
{
\definecolor{dialinecolor}{rgb}{0.000000, 0.000000, 0.000000}
\pgfsetfillcolor{dialinecolor}
\pgfsetarrowsend{to}
\definecolor{dialinecolor}{rgb}{0.000000, 0.000000, 0.000000}
\pgfsetstrokecolor{dialinecolor}
\draw (11.700000\du,6.750000\du)--(12.650000\du,8.300000\du);
}
\definecolor{dialinecolor}{rgb}{0.000000, 0.000000, 0.000000}
\pgfsetstrokecolor{dialinecolor}
\node[anchor=west] at (13.000000\du,14.300000\du){$I^-$};
\definecolor{dialinecolor}{rgb}{0.000000, 0.000000, 0.000000}
\pgfsetstrokecolor{dialinecolor}
\node[anchor=west] at (17.800000\du,14.300000\du){$I$};
\pgfsetlinewidth{0.050000\du}
\pgfsetdash{}{0pt}
\pgfsetdash{}{0pt}
\pgfsetbuttcap
{
\definecolor{dialinecolor}{rgb}{0.000000, 0.000000, 0.000000}
\pgfsetfillcolor{dialinecolor}
\pgfsetarrowsend{to}
\definecolor{dialinecolor}{rgb}{0.000000, 0.000000, 0.000000}
\pgfsetstrokecolor{dialinecolor}
\draw (16.750000\du,9.100000\du)--(15.700000\du,8.000000\du);
}
\definecolor{dialinecolor}{rgb}{0.000000, 0.000000, 0.000000}
\pgfsetstrokecolor{dialinecolor}
\node[anchor=west] at (15.0000\du,9.800000\du){$O_2^*(\lambda_s^{(4)}, \beta = \infty)$};
\end{tikzpicture}}
\caption{No contradiction: {this is the only feasible scenario and conjecture follows from the fact that $O_2^*$ is smaller than $O_1^*$ for $S_p \in I^-$} } \label{fig:nocontd}
  \end{minipage}\hfill
  \begin{minipage}[b]{0.4\textwidth}
	\resizebox{1 \textwidth}{!}{
\ifx\du\undefined
  \newlength{\du}
\fi
\setlength{\du}{15\unitlength}
\begin{tikzpicture}
\pgftransformxscale{1.000000}
\pgftransformyscale{-1.000000}
\definecolor{dialinecolor}{rgb}{0.000000, 0.000000, 0.000000}
\pgfsetstrokecolor{dialinecolor}
\definecolor{dialinecolor}{rgb}{1.000000, 1.000000, 1.000000}
\pgfsetfillcolor{dialinecolor}
\pgfsetlinewidth{0.100000\du}
\pgfsetdash{}{0pt}
\pgfsetdash{}{0pt}
\pgfsetbuttcap
{
\definecolor{dialinecolor}{rgb}{0.000000, 0.000000, 0.000000}
\pgfsetfillcolor{dialinecolor}
\pgfsetarrowsend{stealth}
\definecolor{dialinecolor}{rgb}{0.000000, 0.000000, 0.000000}
\pgfsetstrokecolor{dialinecolor}
\draw (18.900000\du,19.600000\du)--(36.450000\du,19.650000\du);
}
\pgfsetlinewidth{0.100000\du}
\pgfsetdash{}{0pt}
\pgfsetdash{}{0pt}
\pgfsetbuttcap
{
\definecolor{dialinecolor}{rgb}{0.000000, 0.000000, 0.000000}
\pgfsetfillcolor{dialinecolor}
\pgfsetarrowsend{stealth}
\definecolor{dialinecolor}{rgb}{0.000000, 0.000000, 0.000000}
\pgfsetstrokecolor{dialinecolor}
\draw (18.900000\du,19.550000\du)--(18.850000\du,6.500000\du);
}
\pgfsetlinewidth{0.100000\du}
\pgfsetdash{{1.000000\du}{1.000000\du}}{0\du}
\pgfsetdash{{0.400000\du}{0.400000\du}}{0\du}
\pgfsetbuttcap
{
\definecolor{dialinecolor}{rgb}{0.000000, 0.000000, 0.000000}
\pgfsetfillcolor{dialinecolor}
\definecolor{dialinecolor}{rgb}{0.000000, 0.000000, 0.000000}
\pgfsetstrokecolor{dialinecolor}
\draw (28.900000\du,19.550000\du)--(28.900000\du,8.650000\du);
}
\pgfsetlinewidth{0.100000\du}
\pgfsetdash{{0.400000\du}{0.400000\du}}{0\du}
\pgfsetdash{{0.400000\du}{0.400000\du}}{0\du}
\pgfsetbuttcap
{
\definecolor{dialinecolor}{rgb}{0.000000, 0.000000, 0.000000}
\pgfsetfillcolor{dialinecolor}
\definecolor{dialinecolor}{rgb}{0.000000, 0.000000, 0.000000}
\pgfsetstrokecolor{dialinecolor}
\draw (34.600000\du,19.600000\du)--(34.715000\du,8.605000\du);
}
\pgfsetlinewidth{0.100000\du}
\pgfsetdash{}{0pt}
\pgfsetdash{}{0pt}
\pgfsetbuttcap
{
\definecolor{dialinecolor}{rgb}{0.000000, 0.000000, 0.000000}
\pgfsetfillcolor{dialinecolor}
\definecolor{dialinecolor}{rgb}{0.000000, 0.000000, 0.000000}
\pgfsetstrokecolor{dialinecolor}
\pgfpathmoveto{\pgfpoint{28.900442\du}{12.199823\du}}
\pgfpatharc{249}{205}{14.555450\du and 14.555450\du}
\pgfusepath{stroke}
}
\pgfsetlinewidth{0.100000\du}
\pgfsetdash{}{0pt}
\pgfsetdash{}{0pt}
\pgfsetmiterjoin
\pgfsetbuttcap
{
\definecolor{dialinecolor}{rgb}{0.000000, 0.000000, 1.000000}
\pgfsetfillcolor{dialinecolor}
\definecolor{dialinecolor}{rgb}{0.000000, 0.000000, 1.000000}
\pgfsetstrokecolor{dialinecolor}
\pgfpathmoveto{\pgfpoint{21.050000\du}{19.600000\du}}
\pgfpathcurveto{\pgfpoint{21.350000\du}{14.800000\du}}{\pgfpoint{30.200000\du}{12.150000\du}}{\pgfpoint{34.600000\du}{11.550000\du}}
\pgfusepath{stroke}
}
\pgfsetlinewidth{0.100000\du}
\pgfsetdash{}{0pt}
\pgfsetdash{}{0pt}
\pgfsetbuttcap
{
\definecolor{dialinecolor}{rgb}{0.000000, 0.000000, 0.000000}
\pgfsetfillcolor{dialinecolor}
\definecolor{dialinecolor}{rgb}{0.000000, 0.000000, 0.000000}
\pgfsetstrokecolor{dialinecolor}
\draw (28.950000\du,12.250000\du)--(34.800000\du,11.500000\du);
}
\definecolor{dialinecolor}{rgb}{0.000000, 0.000000, 0.000000}
\pgfsetstrokecolor{dialinecolor}
\node[anchor=west] at (17.300000\du,6.100000\du){\large{$O_1^*$ or $O_2^*$}};
\definecolor{dialinecolor}{rgb}{0.000000, 0.000000, 0.000000}
\pgfsetstrokecolor{dialinecolor}
\node[anchor=west] at (37.250000\du,19.900000\du){\large{$S_p$}};
\definecolor{dialinecolor}{rgb}{0.000000, 0.000000, 0.000000}
\pgfsetstrokecolor{dialinecolor}
\node[anchor=west] at (20.800000\du,20.450000\du){\large{$\hat{S}_p$}};
\definecolor{dialinecolor}{rgb}{0.000000, 0.000000, 0.000000}
\pgfsetstrokecolor{dialinecolor}
\node[anchor=west] at (28.700000\du,20.450000\du){\large{$I_l$}};
\definecolor{dialinecolor}{rgb}{0.000000, 0.000000, 0.000000}
\pgfsetstrokecolor{dialinecolor}
\node[anchor=west] at (34.450000\du,20.300000\du){\large{$I_u$}};
\pgfsetlinewidth{0.050000\du}
\pgfsetdash{}{0pt}
\pgfsetdash{}{0pt}
\pgfsetbuttcap
{
\definecolor{dialinecolor}{rgb}{0.000000, 0.000000, 0.000000}
\pgfsetfillcolor{dialinecolor}
\pgfsetarrowsstart{to}
\pgfsetarrowsend{to}
\definecolor{dialinecolor}{rgb}{0.000000, 0.000000, 0.000000}
\pgfsetstrokecolor{dialinecolor}
\draw (21.100000\du,21.000000\du)--(28.850000\du,21.000000\du);
}
\pgfsetlinewidth{0.050000\du}
\pgfsetdash{}{0pt}
\pgfsetdash{}{0pt}
\pgfsetbuttcap
{
\definecolor{dialinecolor}{rgb}{0.000000, 0.000000, 0.000000}
\pgfsetfillcolor{dialinecolor}
\pgfsetarrowsstart{to}
\pgfsetarrowsend{to}
\definecolor{dialinecolor}{rgb}{0.000000, 0.000000, 0.000000}
\pgfsetstrokecolor{dialinecolor}
\draw (29.290000\du,21.045451\du)--(34.250000\du,21.050000\du);
}
\definecolor{dialinecolor}{rgb}{0.000000, 0.000000, 0.000000}
\pgfsetstrokecolor{dialinecolor}
\node[anchor=west] at (24.150000\du,21.950000\du){\large{$I^-$}};
\definecolor{dialinecolor}{rgb}{0.000000, 0.000000, 0.000000}
\pgfsetstrokecolor{dialinecolor}
\node[anchor=west] at (31.750000\du,22.050000\du){\large{$I$}};
\pgfsetlinewidth{0.050000\du}
\pgfsetdash{}{0pt}
\pgfsetdash{}{0pt}
\pgfsetbuttcap
{
\definecolor{dialinecolor}{rgb}{0.000000, 0.000000, 0.000000}
\pgfsetfillcolor{dialinecolor}
\pgfsetarrowsend{to}
\definecolor{dialinecolor}{rgb}{0.000000, 0.000000, 0.000000}
\pgfsetstrokecolor{dialinecolor}
\draw (23.900000\du,10.800000\du)--(25.750000\du,13.500000\du);
}
\definecolor{dialinecolor}{rgb}{0.000000, 0.000000, 0.000000}
\pgfsetstrokecolor{dialinecolor}
\node[anchor=west] at (21.50000\du,10.300000\du){\large{$ O_1^*(\lambda_s^{(2)}, 0)$}};
\definecolor{dialinecolor}{rgb}{0.000000, 0.000000, 0.000000}
\pgfsetstrokecolor{dialinecolor}
\node[anchor=west] at (29.250000\du,16.550000\du){\large{$ O_2^*(\lambda_s^{(4)}, \beta = \infty)$}};
\pgfsetlinewidth{0.050000\du}
\pgfsetdash{}{0pt}
\pgfsetdash{}{0pt}
\pgfsetbuttcap
{
\definecolor{dialinecolor}{rgb}{0.000000, 0.000000, 0.000000}
\pgfsetfillcolor{dialinecolor}
\pgfsetarrowsend{to}
\definecolor{dialinecolor}{rgb}{0.000000, 0.000000, 0.000000}
\pgfsetstrokecolor{dialinecolor}
\draw (29.700000\du,16.050000\du)--(27.899754\du,13.639754\du);
}
\pgfsetlinewidth{0.050000\du}
\pgfsetdash{}{0pt}
\pgfsetdash{}{0pt}
\pgfsetbuttcap
{
\definecolor{dialinecolor}{rgb}{0.000000, 0.000000, 0.000000}
\pgfsetfillcolor{dialinecolor}
\pgfsetarrowsend{to}
\definecolor{dialinecolor}{rgb}{0.000000, 0.000000, 0.000000}
\pgfsetstrokecolor{dialinecolor}
\draw (31.550000\du,10.050000\du)--(30.314754\du,11.744754\du);
}
\definecolor{dialinecolor}{rgb}{0.000000, 0.000000, 0.000000}
\pgfsetstrokecolor{dialinecolor}
\node[anchor=west] at (30.115000\du,9.650000\du){\large{$ O_1^*(\lambda_s^{(1)}, \beta^{(1)})$}};
\end{tikzpicture}}
\caption{{ Slope of $O_2^*$ is higher than that of $O_1^*$ near $\hat{S}_p$ but both the slopes are positive.} Contradiction by $O_2^*(\lambda_s^i,\infty) < O_1^*(\lambda_s^f,\beta^f)$ at $\hat{S}_p+\epsilon$}\label{fig:2cuts}
  \end{minipage}
\end{figure}

\end{proof}

{
 The optimal objective of problem P1 is more than that of P2 for $S_p \in I$, i.e., $O_1^*(\lambda_s^f,\beta^f) > O_2^*(\lambda_s^i,\beta^i)$ (see \cite[page 26]{Sudhir_TR}). Thus, the optimal solution of P0 is given by that of P1 for $S_p \in I$. {The concave increasing nature of $O_2^*$ ensures that $O_2^*$ cannot cross $O_1^*$ more than once as this results in a non-concave function (See figure \ref{fig:non_concave}, $O_1^*$ and $O_2^*$ are in black and blue colors respectively in all subsequent figures).  }

The following (Theorem 2) result determines the nature of {{}the} optimal objective of problem P0 in service level range $I^-$ by exploiting the increasing concave nature of optimal objective function of problem P2 as identified in {Theorem 1}. {The conjecture follows from Theorem 2.}

\begin{thm}
The optimal solution of P0 is given by {{}the} optimal solution of P1 for $S_p \in I^-$.
\end{thm}

\begin{proof}

The optimal objective of problem P1, $O_1^*$, is increasing concave in $I^-$ and linearly increasing in $I$ as stated in Theorem 1; {{}$O_2^*$ cannot cross $O_1^*$ more than once to ensure concavity as discussed above.} Thus, four scenarios can arise for optimal objective of problem P2, $O_2^*$, for service level range $I^-$ as shown in Figures \ref{fig:nocontd}, \ref{fig:2cuts}, \ref{fig:1cut}, and \ref{fig:infeasible} depending on the slope of $O_1^*$ and $O_2^*$ near $\hat{S}_p$.

\begin{figure}[htbp]
  \centering
  \begin{minipage}[b]{0.4\textwidth}
\resizebox{1 \textwidth}{!}{
\ifx\du\undefined
  \newlength{\du}
\fi
\setlength{\du}{15\unitlength}
\begin{tikzpicture}
\pgftransformxscale{1.000000}
\pgftransformyscale{-1.000000}
\definecolor{dialinecolor}{rgb}{0.000000, 0.000000, 0.000000}
\pgfsetstrokecolor{dialinecolor}
\definecolor{dialinecolor}{rgb}{1.000000, 1.000000, 1.000000}
\pgfsetfillcolor{dialinecolor}
\pgfsetlinewidth{0.100000\du}
\pgfsetdash{}{0pt}
\pgfsetdash{}{0pt}
\pgfsetbuttcap
{
\definecolor{dialinecolor}{rgb}{0.000000, 0.000000, 0.000000}
\pgfsetfillcolor{dialinecolor}
\pgfsetarrowsend{stealth}
\definecolor{dialinecolor}{rgb}{0.000000, 0.000000, 0.000000}
\pgfsetstrokecolor{dialinecolor}
\draw (6.265000\du,15.535000\du)--(19.065000\du,15.485000\du);
}
\pgfsetlinewidth{0.100000\du}
\pgfsetdash{}{0pt}
\pgfsetdash{}{0pt}
\pgfsetbuttcap
{
\definecolor{dialinecolor}{rgb}{0.000000, 0.000000, 0.000000}
\pgfsetfillcolor{dialinecolor}
\pgfsetarrowsend{stealth}
\definecolor{dialinecolor}{rgb}{0.000000, 0.000000, 0.000000}
\pgfsetstrokecolor{dialinecolor}
\draw (6.315000\du,15.535000\du)--(6.315000\du,4.885000\du);
}
\pgfsetlinewidth{0.000000\du}
\pgfsetdash{{1.000000\du}{1.000000\du}}{0\du}
\pgfsetdash{{0.200000\du}{0.200000\du}}{0\du}
\pgfsetbuttcap
{
\definecolor{dialinecolor}{rgb}{0.000000, 0.000000, 0.000000}
\pgfsetfillcolor{dialinecolor}
\definecolor{dialinecolor}{rgb}{0.000000, 0.000000, 0.000000}
\pgfsetstrokecolor{dialinecolor}
\draw (10.415000\du,4.985000\du)--(10.465000\du,15.435000\du);
}
\pgfsetlinewidth{0.000000\du}
\pgfsetdash{{0.200000\du}{0.200000\du}}{0\du}
\pgfsetdash{{0.200000\du}{0.200000\du}}{0\du}
\pgfsetbuttcap
{
\definecolor{dialinecolor}{rgb}{0.000000, 0.000000, 0.000000}
\pgfsetfillcolor{dialinecolor}
\definecolor{dialinecolor}{rgb}{0.000000, 0.000000, 0.000000}
\pgfsetstrokecolor{dialinecolor}
\draw (17.230239\du,5.025239\du)--(17.280239\du,15.475239\du);
}
\pgfsetlinewidth{0.100000\du}
\pgfsetdash{}{0pt}
\pgfsetdash{}{0pt}
\pgfsetbuttcap
{
\definecolor{dialinecolor}{rgb}{0.000000, 0.000000, 0.000000}
\pgfsetfillcolor{dialinecolor}
\definecolor{dialinecolor}{rgb}{0.000000, 0.000000, 0.000000}
\pgfsetstrokecolor{dialinecolor}
\draw (10.415000\du,9.685000\du)--(17.215000\du,7.735000\du);
}
\pgfsetlinewidth{0.100000\du}
\pgfsetdash{}{0pt}
\pgfsetdash{}{0pt}
\pgfsetbuttcap
{
\definecolor{dialinecolor}{rgb}{0.000000, 0.000000, 0.000000}
\pgfsetfillcolor{dialinecolor}
\definecolor{dialinecolor}{rgb}{0.000000, 0.000000, 0.000000}
\pgfsetstrokecolor{dialinecolor}
\pgfpathmoveto{\pgfpoint{10.515279\du}{9.584628\du}}
\pgfpatharc{217}{184}{11.103314\du and 11.103314\du}
\pgfusepath{stroke}
}
\pgfsetlinewidth{0.100000\du}
\pgfsetdash{}{0pt}
\pgfsetdash{}{0pt}
\pgfsetmiterjoin
\pgfsetbuttcap
{
\definecolor{dialinecolor}{rgb}{0.000000, 0.000000, 0.000000}
\pgfsetfillcolor{dialinecolor}
\definecolor{dialinecolor}{rgb}{0.000000, 0.000000, 1.000000}
\pgfsetstrokecolor{dialinecolor}
\pgfpathmoveto{\pgfpoint{8.265000\du}{15.385000\du}}
\pgfpathcurveto{\pgfpoint{5.200000\du}{13.450000\du}}{\pgfpoint{14.065000\du}{8.435000\du}}{\pgfpoint{17.015000\du}{7.835000\du}}
\pgfusepath{stroke}
}
\pgfsetlinewidth{0.050000\du}
\pgfsetdash{}{0pt}
\pgfsetdash{}{0pt}
\pgfsetbuttcap
{
\definecolor{dialinecolor}{rgb}{0.000000, 0.000000, 0.000000}
\pgfsetfillcolor{dialinecolor}
\pgfsetarrowsstart{to}
\pgfsetarrowsend{to}
\definecolor{dialinecolor}{rgb}{0.000000, 0.000000, 0.000000}
\pgfsetstrokecolor{dialinecolor}
\draw (8.415000\du,16.885000\du)--(10.465000\du,16.885000\du);
}
\pgfsetlinewidth{0.050000\du}
\pgfsetdash{}{0pt}
\pgfsetdash{}{0pt}
\pgfsetbuttcap
{
\definecolor{dialinecolor}{rgb}{0.000000, 0.000000, 0.000000}
\pgfsetfillcolor{dialinecolor}
\pgfsetarrowsstart{to}
\pgfsetarrowsend{to}
\definecolor{dialinecolor}{rgb}{0.000000, 0.000000, 0.000000}
\pgfsetstrokecolor{dialinecolor}
\draw (10.865000\du,16.835000\du)--(17.515000\du,16.835000\du);
}
\definecolor{dialinecolor}{rgb}{0.000000, 0.000000, 0.000000}
\pgfsetstrokecolor{dialinecolor}
\node[anchor=west] at (7.865000\du,16.285000\du){$\hat{S}_p$};
\definecolor{dialinecolor}{rgb}{0.000000, 0.000000, 0.000000}
\pgfsetstrokecolor{dialinecolor}
\node[anchor=west] at (10.415000\du,16.285000\du){$I_l$};
\definecolor{dialinecolor}{rgb}{0.000000, 0.000000, 0.000000}
\pgfsetstrokecolor{dialinecolor}
\node[anchor=west] at (17.065000\du,16.185000\du){$I_u$};
\definecolor{dialinecolor}{rgb}{0.000000, 0.000000, 0.000000}
\pgfsetstrokecolor{dialinecolor}
\node[anchor=west] at (19.465000\du,15.835000\du){$S_p$};
\definecolor{dialinecolor}{rgb}{0.000000, 0.000000, 0.000000}
\pgfsetstrokecolor{dialinecolor}
\node[anchor=west] at (4.58000\du,4.585000\du){$O_1^*$ or $O_2^*$};
\definecolor{dialinecolor}{rgb}{0.000000, 0.000000, 0.000000}
\pgfsetstrokecolor{dialinecolor}
\node[anchor=west] at (11.865000\du,6.8435000\du){$O_1^*(\lambda_s^{(1)}, \beta^{(1)})$};
\pgfsetlinewidth{0.050000\du}
\pgfsetdash{}{0pt}
\pgfsetdash{}{0pt}
\pgfsetbuttcap
{
\definecolor{dialinecolor}{rgb}{0.000000, 0.000000, 0.000000}
\pgfsetfillcolor{dialinecolor}
\pgfsetarrowsend{to}
\definecolor{dialinecolor}{rgb}{0.000000, 0.000000, 0.000000}
\pgfsetstrokecolor{dialinecolor}
\draw (12.465000\du,7.685000\du)--(11.715000\du,9.035000\du);
}
\definecolor{dialinecolor}{rgb}{0.000000, 0.000000, 0.000000}
\pgfsetstrokecolor{dialinecolor}
\node[anchor=west] at (6.5000\du,9.2000\du){$O_1^*(\lambda_s^{(2)}, 0)$};
\pgfsetlinewidth{0.050000\du}
\pgfsetdash{}{0pt}
\pgfsetdash{}{0pt}
\pgfsetbuttcap
{
\definecolor{dialinecolor}{rgb}{0.000000, 0.000000, 0.000000}
\pgfsetfillcolor{dialinecolor}
\pgfsetarrowsend{to}
\definecolor{dialinecolor}{rgb}{0.000000, 0.000000, 0.000000}
\pgfsetstrokecolor{dialinecolor}
\draw (7.965000\du,10.135000\du)--(8.915000\du,11.685000\du);
}
\definecolor{dialinecolor}{rgb}{0.000000, 0.000000, 0.000000}
\pgfsetstrokecolor{dialinecolor}
\node[anchor=west] at (9.265000\du,17.685000\du){$I^-$};
\definecolor{dialinecolor}{rgb}{0.000000, 0.000000, 0.000000}
\pgfsetstrokecolor{dialinecolor}
\node[anchor=west] at (14.065000\du,17.685000\du){$I$};
\pgfsetlinewidth{0.050000\du}
\pgfsetdash{}{0pt}
\pgfsetdash{}{0pt}
\pgfsetbuttcap
{
\definecolor{dialinecolor}{rgb}{0.000000, 0.000000, 0.000000}
\pgfsetfillcolor{dialinecolor}
\pgfsetarrowsend{to}
\definecolor{dialinecolor}{rgb}{0.000000, 0.000000, 0.000000}
\pgfsetstrokecolor{dialinecolor}
\draw (13.015000\du,12.485000\du)--(11.965000\du,11.385000\du);
}
\definecolor{dialinecolor}{rgb}{0.000000, 0.000000, 0.000000}
\pgfsetstrokecolor{dialinecolor}
\node[anchor=west] at (11.000\du,13.185000\du){$O_2^*(\lambda_s^{(4)}, \beta = \infty)$};
\end{tikzpicture}}
\caption{{ Slope of $O_2^*$ is negative and higher than that of $O_1^*$ near $\hat{S}_p$. Contradiction by infeasibility of $O_1^*$ (equivalently problem P1) for $S_p < \hat{S}_p$.} }\label{fig:1cut}
  \end{minipage}\hfill
  \begin{minipage}[b]{0.4\textwidth}
	\resizebox{1 \textwidth}{!}{
\ifx\du\undefined
  \newlength{\du}
\fi
\setlength{\du}{15\unitlength}
\begin{tikzpicture}
\pgftransformxscale{1.000000}
\pgftransformyscale{-1.000000}
\definecolor{dialinecolor}{rgb}{0.000000, 0.000000, 0.000000}
\pgfsetstrokecolor{dialinecolor}
\definecolor{dialinecolor}{rgb}{1.000000, 1.000000, 1.000000}
\pgfsetfillcolor{dialinecolor}
\pgfsetlinewidth{0.100000\du}
\pgfsetdash{}{0pt}
\pgfsetdash{}{0pt}
\pgfsetbuttcap
{
\definecolor{dialinecolor}{rgb}{0.000000, 0.000000, 0.000000}
\pgfsetfillcolor{dialinecolor}
\pgfsetarrowsend{stealth}
\definecolor{dialinecolor}{rgb}{0.000000, 0.000000, 0.000000}
\pgfsetstrokecolor{dialinecolor}
\draw (11.680000\du,19.525000\du)--(24.480000\du,19.475000\du);
}
\pgfsetlinewidth{0.100000\du}
\pgfsetdash{}{0pt}
\pgfsetdash{}{0pt}
\pgfsetbuttcap
{
\definecolor{dialinecolor}{rgb}{0.000000, 0.000000, 0.000000}
\pgfsetfillcolor{dialinecolor}
\pgfsetarrowsend{stealth}
\definecolor{dialinecolor}{rgb}{0.000000, 0.000000, 0.000000}
\pgfsetstrokecolor{dialinecolor}
\draw (11.730000\du,19.525000\du)--(11.730000\du,8.875000\du);
}
\pgfsetlinewidth{0.000000\du}
\pgfsetdash{{1.000000\du}{1.000000\du}}{0\du}
\pgfsetdash{{0.200000\du}{0.200000\du}}{0\du}
\pgfsetbuttcap
{
\definecolor{dialinecolor}{rgb}{0.000000, 0.000000, 0.000000}
\pgfsetfillcolor{dialinecolor}
\definecolor{dialinecolor}{rgb}{0.000000, 0.000000, 0.000000}
\pgfsetstrokecolor{dialinecolor}
\draw (15.830000\du,8.975000\du)--(15.880000\du,19.425000\du);
}
\pgfsetlinewidth{0.000000\du}
\pgfsetdash{{0.200000\du}{0.200000\du}}{0\du}
\pgfsetdash{{0.200000\du}{0.200000\du}}{0\du}
\pgfsetbuttcap
{
\definecolor{dialinecolor}{rgb}{0.000000, 0.000000, 0.000000}
\pgfsetfillcolor{dialinecolor}
\definecolor{dialinecolor}{rgb}{0.000000, 0.000000, 0.000000}
\pgfsetstrokecolor{dialinecolor}
\draw (22.645239\du,9.015239\du)--(22.695239\du,19.465239\du);
}
\pgfsetlinewidth{0.100000\du}
\pgfsetdash{}{0pt}
\pgfsetdash{}{0pt}
\pgfsetbuttcap
{
\definecolor{dialinecolor}{rgb}{0.000000, 0.000000, 0.000000}
\pgfsetfillcolor{dialinecolor}
\definecolor{dialinecolor}{rgb}{0.000000, 0.000000, 0.000000}
\pgfsetstrokecolor{dialinecolor}
\draw (15.830000\du,13.675000\du)--(22.630000\du,11.725000\du);
}
\pgfsetlinewidth{0.100000\du}
\pgfsetdash{}{0pt}
\pgfsetdash{}{0pt}
\pgfsetbuttcap
{
\definecolor{dialinecolor}{rgb}{0.000000, 0.000000, 0.000000}
\pgfsetfillcolor{dialinecolor}
\definecolor{dialinecolor}{rgb}{0.000000, 0.000000, 0.000000}
\pgfsetstrokecolor{dialinecolor}
\pgfpathmoveto{\pgfpoint{15.930279\du}{13.574628\du}}
\pgfpatharc{217}{184}{11.103314\du and 11.103314\du}
\pgfusepath{stroke}
}
\pgfsetlinewidth{0.100000\du}
\pgfsetdash{}{0pt}
\pgfsetdash{}{0pt}
\pgfsetmiterjoin
\pgfsetbuttcap
{
\definecolor{dialinecolor}{rgb}{0.000000, 0.000000, 0.000000}
\pgfsetfillcolor{dialinecolor}
\definecolor{dialinecolor}{rgb}{0.000000, 0.000000, 1.000000}
\pgfsetstrokecolor{dialinecolor}
\pgfpathmoveto{\pgfpoint{13.660000\du}{19.375000\du}}
\pgfpathcurveto{\pgfpoint{12.900000\du}{9.300000\du}}{\pgfpoint{23.100000\du}{10.200000\du}}{\pgfpoint{22.530000\du}{11.825000\du}}
\pgfusepath{stroke}
}
\pgfsetlinewidth{0.050000\du}
\pgfsetdash{}{0pt}
\pgfsetdash{}{0pt}
\pgfsetbuttcap
{
\definecolor{dialinecolor}{rgb}{0.000000, 0.000000, 0.000000}
\pgfsetfillcolor{dialinecolor}
\pgfsetarrowsstart{to}
\pgfsetarrowsend{to}
\definecolor{dialinecolor}{rgb}{0.000000, 0.000000, 0.000000}
\pgfsetstrokecolor{dialinecolor}
\draw (13.830000\du,20.875000\du)--(15.880000\du,20.875000\du);
}
\pgfsetlinewidth{0.050000\du}
\pgfsetdash{}{0pt}
\pgfsetdash{}{0pt}
\pgfsetbuttcap
{
\definecolor{dialinecolor}{rgb}{0.000000, 0.000000, 0.000000}
\pgfsetfillcolor{dialinecolor}
\pgfsetarrowsstart{to}
\pgfsetarrowsend{to}
\definecolor{dialinecolor}{rgb}{0.000000, 0.000000, 0.000000}
\pgfsetstrokecolor{dialinecolor}
\draw (16.280000\du,20.825000\du)--(22.930000\du,20.825000\du);
}
\definecolor{dialinecolor}{rgb}{0.000000, 0.000000, 0.000000}
\pgfsetstrokecolor{dialinecolor}
\node[anchor=west] at (13.280000\du,20.275000\du){$\hat{S}_p$};
\definecolor{dialinecolor}{rgb}{0.000000, 0.000000, 0.000000}
\pgfsetstrokecolor{dialinecolor}
\node[anchor=west] at (15.830000\du,20.275000\du){$I_l$};
\definecolor{dialinecolor}{rgb}{0.000000, 0.000000, 0.000000}
\pgfsetstrokecolor{dialinecolor}
\node[anchor=west] at (22.480000\du,20.175000\du){$I_u$};
\definecolor{dialinecolor}{rgb}{0.000000, 0.000000, 0.000000}
\pgfsetstrokecolor{dialinecolor}
\node[anchor=west] at (24.880000\du,19.825000\du){$S_p$};
\definecolor{dialinecolor}{rgb}{0.000000, 0.000000, 0.000000}
\pgfsetstrokecolor{dialinecolor}
\node[anchor=west] at (10.0000\du,8.575000\du){$O_1^*$ or $O_2^*$};
\definecolor{dialinecolor}{rgb}{0.000000, 0.000000, 0.000000}
\pgfsetstrokecolor{dialinecolor}
\node[anchor=west] at (18.00000\du,14.725000\du){$O_1^*(\lambda_s^{(1)}, \beta^{(1)})$};
\pgfsetlinewidth{0.050000\du}
\pgfsetdash{}{0pt}
\pgfsetdash{}{0pt}
\pgfsetbuttcap
{
\definecolor{dialinecolor}{rgb}{0.000000, 0.000000, 0.000000}
\pgfsetfillcolor{dialinecolor}
\pgfsetarrowsend{to}
\definecolor{dialinecolor}{rgb}{0.000000, 0.000000, 0.000000}
\pgfsetstrokecolor{dialinecolor}
\draw (19.250000\du,14.100000\du)--(18.450000\du,13.250000\du);
}
\definecolor{dialinecolor}{rgb}{0.000000, 0.000000, 0.000000}
\pgfsetstrokecolor{dialinecolor}
\node[anchor=west] at (12.410000\du,13.45000\du){$O_1^*(\lambda_s^{(2)}, 0)$};
\pgfsetlinewidth{0.050000\du}
\pgfsetdash{}{0pt}
\pgfsetdash{}{0pt}
\pgfsetbuttcap
{
\definecolor{dialinecolor}{rgb}{0.000000, 0.000000, 0.000000}
\pgfsetfillcolor{dialinecolor}
\pgfsetarrowsend{to}
\definecolor{dialinecolor}{rgb}{0.000000, 0.000000, 0.000000}
\pgfsetstrokecolor{dialinecolor}
\draw (13.380000\du,14.125000\du)--(14.330000\du,15.675000\du);
}
\definecolor{dialinecolor}{rgb}{0.000000, 0.000000, 0.000000}
\pgfsetstrokecolor{dialinecolor}
\node[anchor=west] at (14.680000\du,21.675000\du){$I^-$};
\definecolor{dialinecolor}{rgb}{0.000000, 0.000000, 0.000000}
\pgfsetstrokecolor{dialinecolor}
\node[anchor=west] at (19.480000\du,21.675000\du){$I$};
\pgfsetlinewidth{0.050000\du}
\pgfsetdash{}{0pt}
\pgfsetdash{}{0pt}
\pgfsetbuttcap
{
\definecolor{dialinecolor}{rgb}{0.000000, 0.000000, 0.000000}
\pgfsetfillcolor{dialinecolor}
\pgfsetarrowsend{to}
\definecolor{dialinecolor}{rgb}{0.000000, 0.000000, 0.000000}
\pgfsetstrokecolor{dialinecolor}
\draw (17.950000\du,10.300000\du)--(17.500000\du,11.500000\du);
}
\definecolor{dialinecolor}{rgb}{0.000000, 0.000000, 0.000000}
\pgfsetstrokecolor{dialinecolor}
\node[anchor=west] at (16.0000\du,9.8000\du){$O_2^*(\lambda_s^{(4)}, \beta = \infty)$};
\end{tikzpicture}}
\caption{ Contradiction by the fact that $O_2^*(\lambda_s^i,\infty) < O_1^*(\lambda_s^f,\beta^f)$ for $S_p \in I$}\label{fig:infeasible}
  \end{minipage}
\end{figure}

In Figure \ref{fig:2cuts}, slope of $O_2^*$ is higher than that of $O_1^*$ near $\hat{S}_p$. But such a scenario is not possible by the fact that $O_2^*(\lambda_s^i,\infty) < O_1^*(\lambda_s^f,\beta^f)$ at $\hat{S}_p+\epsilon$, where $\epsilon$ is a small positive number~\cite[page 26]{Sudhir_TR}. Figure \ref{fig:1cut} has the negative slope for $O_2^*$ near $\hat{S}_p$ and this scenario is not possible as it contradicts the feasibility requirements (i.e. $S_p \geq \hat{S}_p$); and the Figure \ref{fig:infeasible} is also not possible as it contradicts the fact that optimal objective of problem P1 is more than that of  P2 for $S_p \in I$, i.e., $O_1^*(\lambda_s^f,\beta^f) > O_2^*(\lambda_s^i,\beta^i)$ (see \cite[page 26]{Sudhir_TR}).
Thus, the only possibility is Figure \ref{fig:nocontd} where the slope of $O_2^*$ is less than $O_1^*$ near $\hat{S}_p$ and hence $O_2^*(\lambda_s^i,\infty) < O_1^*(\lambda_s^f,\beta^f)$ for $S_p \in I^-$. And the theorem follows.

%

\end{proof}

We now discuss the implications of this conjecture on a finite step algorithm to compute the optimal solution for problem \textbf{P0}. 

\subsection{Validity of algorithm}\label{algo_validity}
Since the conjecture is indeed true, it validates the finite step algorithm of \cite{Sudhir_standard_style} to compute the optimal operating parameters for optimization problem \textbf{P0}.  We briefly outline the algorithm below and reproduce the same. 

The input parameters for the algorithm are $\lambda_p, \mu, \sigma, a, b, c $ and the promised QoS level for primary class customers $S_p$. The algorithm accounts for the solution of problem \textbf{P0}  based on the range of $S_p$. If the value of $S_p$ is smaller than the given threshold $\hat{S}_p$, the algorithm results into in-feasible solution in Step 2. Otherwise, it appropriately computes the optimal arrival rate for secondary class customers $(\lambda_s^*)$ and the optimal scheduling parameter ($\beta^*$) in steps 3-6. Finally, algorithm finds the optimal price ($\theta^*$) and the service level for secondary class customers ($S_s^*$) in Step 7. This algorithm is used to compute the \textit{revenue maximal} NE in a non-cooperative game between service-provider and the user-set as discussed in Section \ref{sec:sysdes}.

{\small
\begin{algorithm}                     
 \caption{Finite step algorithm for solving problem \textbf{P0}}
\label{findme}                          
\begin{algorithmic} [1]                   
\REQUIRE   $\lambda_{p}$, $\mu$, $\sigma$, $a$, $b$, $c$ and $S_{p}$ 
    \STATE Define $\psi= \left[1+\sigma^{2}\mu^{2}\right]/2$.
    
        \STATE If either $\displaystyle S_{p} \leq \hat{S}_{p} \equiv \frac{\lambda_{p} \psi}{\mu \left[\mu-\lambda_{p} \right]}$ or $\displaystyle \frac{a}{c} \leq \frac{\lambda_{p}}{\mu^{2}}\psi$, then there does not exist a feasible solution. Assign $\lambda_{s}^{*}=0$ and Stop. Else, go to the Step 3.

\STATE If $\displaystyle \frac{\mu-\lambda_{p}}{\mu \lambda_{p}} \leq \frac{ a \lambda_{p} - c \psi}{ 2 \mu \lambda_{p}^{2}+ c \psi (\mu+\lambda_{p})}$, then assign $J_{\ell}=\infty$ and go to the Step 4.
Else, find $\lambda_{s}^{(3)}$ the unique root of the cubic $\tilde{G}(\lambda_{s})$ which lies in the interval
$\left(0, \mu -\lambda_{p}\right)$ where $        \tilde{G}(\lambda_{s})$ is
        $$ 2 \mu \lambda_{s}^{3}- \left[a \mu +
        c \psi + 4 \mu^{2} \right]\lambda_{s}^{2}+ 2 \mu \left[a \mu+c
        \psi+\mu^{2} \right]\lambda_{s}- \mu \left[a \mu^{2}-c \psi
        \lambda_{p}\right].
        $$
Calculate $J_{\ell}=\frac{\psi \lambda_{3}}{[\mu-\lambda_{s}^{(3)}][\mu-\lambda_{3}]}$ and define an interval
$J=(J_{\ell}, \infty)$ where $\lambda_{3}=\lambda_{p}+\lambda_{s}^{(3)}$. If $S_{p} \in J$, then assign $\lambda_{s}^{*}=\lambda_{s}^{(3)}$, $\beta^{*}=\infty$ and directly go to Step 7. Else, go to the Step 4.    
\STATE If $\displaystyle \frac{a}{c} \leq \frac{\lambda_{p}\left(2\mu- \lambda_{p}\right)}{\mu \left(\mu
-\lambda_{p}\right)^{2}}\psi$, then define an interval  $J^{-}=(\hat{S}_{p}, J_{\ell} ]$ when $J_{\ell}$ is finite
otherwise take $J^{-}=(\hat{S}_{p}, \infty)$. Assign $ \lambda_{s}^{*} =\displaystyle \frac{1}{2 S_{p}}\left[ S_{p} \left[2 \mu - \lambda_{p} \right]+ \psi - \sqrt{\left[S_{p} \lambda_{p} + \psi \right]^{2}+4 \mu \psi
S_{p} } \right]$, $\beta^{*}=\infty$ for $S_{p} \in J^{-}$ and directly go to Step 7. Else, go to the Step 5.

\STATE Find $\displaystyle\lambda_{s}^{(1)}$, the unique root of the cubic
$G(\lambda_{s})$ in the interval $\left(0, ~\mu-\lambda_{p}
\right)$ with $\phi_{0}=\mu-\lambda_{p}$ and $G(\lambda_{s})=$
$$
 2 \mu \lambda_{s}^{3}- [c \psi + \mu (a+4
\phi_{0})] \lambda_{s}^{2}+ 2 \phi_{0}[c \psi + \mu
(a+\phi_{0})]\lambda_{s} - a \mu \phi_{0}^{2}+c \psi
\lambda_{p}(\mu + \phi_{0}).
$$
Calculate $I_{\ell}=\frac{\psi \lambda_{1}}{\mu
[\mu-\lambda_{p}]}$ and $I_{u}=\frac{\psi \lambda_{1}}{[\mu-
\lambda_{s}^{(1)}][\mu-\lambda_{1}]}$ where
$\lambda_{1}=\lambda_{p}+\lambda_{s}^{(1)}$.

\STATE Find $\displaystyle\lambda_{s}^{(1)}$, the unique root of the cubic
$G(\lambda_{s})$ in the interval $\left(0, ~\mu-\lambda_{p}
\right)$ with $\phi_{0}=\mu-\lambda_{p}$ and $G(\lambda_{s})=$
$$
 2 \mu \lambda_{s}^{3}- [c \psi + \mu (a+4
\phi_{0})] \lambda_{s}^{2}+ 2 \phi_{0}[c \psi + \mu
(a+\phi_{0})]\lambda_{s} - a \mu \phi_{0}^{2}+c \psi
\lambda_{p}(\mu + \phi_{0}).
$$
Calculate $I_{\ell}=\frac{\psi \lambda_{1}}{\mu
[\mu-\lambda_{p}]}$ and $I_{u}=\frac{\psi \lambda_{1}}{[\mu-
\lambda_{s}^{(1)}][\mu-\lambda_{1}]}$ where
$\lambda_{1}=\lambda_{p}+\lambda_{s}^{(1)}$.

\STATE Define intervals: $I^{-}=(\hat{S}_{p}, I_{\ell})$, $I
=[I_{\ell}, I_{u} )$ and $I^{+}=[I_{u}, J_{\ell}]$ when $J_{\ell}$ is finite, otherwise
take $I^{+}$ as $I^{+}=[I_{u}, \infty)$.
\begin{enumerate}[(i)]
\item  If $S_{p} \in I^{-}$, then assign
$\lambda_{s}^{*}=\frac{\mu [\mu-\lambda_{p}] S_{p}}{\psi} -
\lambda_{p}$ and $\beta^{*}=0$

\item If $S_{p} \in I$, then assign
$\lambda_{s}^{*}=\lambda_{s}^{(1)}$ and 
\begin{equation*}
\beta^{*} =\left\{ \begin{array}{ll} \displaystyle
\frac{[\mu-\lambda_{1}] [\mu S_{p} [\mu-\lambda_{p}]-\psi
\lambda_{1}] } {\psi \lambda_{1}^{2}- \mu S_{p}
\lambda_{p}[\mu-\lambda_{1}]} & \hspace{-.8cm}\qquad \mbox{for }
{\frac{\psi \lambda_{1}}{\mu [\mu-\lambda_{p}]} \leq S_{p}} \leq
\frac{\psi
\lambda_{1}}{\mu[\mu-\lambda_{1}]}\\
&\\
\displaystyle \frac{S_{p} \lambda_{s}^{(1)}[\mu-\lambda_{1}] }
{\psi \lambda_{1}- S_{p} [\mu-\lambda_{s}^{(1)}][\mu-\lambda_{1}]}
& \qquad \hspace{-1.1cm}\mbox{for} \frac{\psi \lambda_{1}}{\mu
[\mu-\lambda_{1}]} < S_{p} < \frac{\psi \lambda_{1}}{[\mu-
\lambda_{s}^{(1)}][\mu-\lambda_{1}]}
\end{array} \right.
\end{equation*}

\item If $S_{p} \in I^{+} $, then assign $\lambda_{s}^{*}=
\displaystyle \frac{1}{2 S_{p}}\left[ S_{p} \left[2 \mu -
\lambda_{p} \right]+ \psi - \sqrt{\left[S_{p} \lambda_{p} + \psi
\right]^{2}+4 \mu \psi S_{p} } \right]$ and $\beta^{*}=\infty$.
\end{enumerate}
\STATE If given problem is feasible, the optimum assured service
level to the secondary class customers is $S_{s}^{*}=W_{s}
(\lambda_{s}^{*}, \beta^{*})$ and the optimal unit price charged
to the secondary class customers is $\theta^{*}=\left[a-c
S_{s}^{*}- \lambda_{s}^{*}\right]/b$.

\end{algorithmic}
\end{algorithm}}

\subsection{Unique Solution of Problem P0} \label{UniqueP0}

We argue below that $(\lambda_s^*, \theta^*, S_s^*, \beta^*)$ is the unique solution of optimization problem \textbf{P0}. Once the optimal secondary class mean arrival rate $\lambda_{s}^{*}$ and queue discipline management parameter $\beta^{*}$ are calculated from finite step algorithm, the optimal admission price $\theta^{*}$ and assured service level to secondary class $S_{s}^{*}$ can be computed (see \cite{Sudhir_standard_style}):
$$S_s^* = W_s(\lambda_s^*,\beta^*) \text{ and }\theta^* = [a-cS^*-\lambda_s^*]/b$$
Thus, uniqueness of $\lambda_s^*$ and $\beta^*$ will suffice for unique solution of optimization problem \textbf{P0}. $\lambda_s^*$ and $\beta^*$ are given by Theorem 1, 2, 3 or 4 in \cite{Sudhir_standard_style}. Note that $\beta^*$ is either 0 or $\infty$ or a unique root of  cubic polynomials in the interval $[0, \mu - \lambda_p)$ depending on the value of service level $S_p$. And $\lambda_s^*$ is either a unique root of cubic (see Theorem 1 or 3 in \cite{Sudhir_standard_style}) or a unique expression for given service level $S_p$ (see Theorem 2 or 4 in \cite{Sudhir_standard_style}). Hence, $\lambda_s^*$ and $\beta^*$ are unique and so is the solution of problem \textbf{P0}. 
\section{{Service-provider and user-set strategic interaction}}
\label{sec:sysdes}
We first present the details of non-cooperative game between service provider and the user-set. We further discuss the existence of the continuum of NE and its computation via constrained best response dynamics. We remark that the revenue maximal NE is identified by above finite step algorithm in this continuum of NE. We also illustrate with numerical examples that both players can benefit at such revenue maximal NE by identifying suitable operational decisions, i.e., by choosing an appropriate game.

\subsection{The non-cooperative game} 

Consider a single shot two player non-cooperative game between service-provider (firm) and the user-set (market). As part of $SLA_p$ the service-provider needs to ensure that primary class customers mean waiting time should not exceed a specified level $S_p$, i.e., $W_p(\lambda_s, \beta) \leq S_p$. The service-provider also needs to ensure that the mean waiting time of secondary class customers, $W_s(\lambda_s, \beta)$ does not exceed $S_s$. Unit admission price, $\theta$, service level to secondary class customers, $S_s$, and scheduling policy across classes, $\beta$ are decided by the service-provider, player $1$. User-set, player 2, reacts according to a linear demand function ($\lambda_s \leq a-b\theta - c S_s$) for a given admission price, $\theta$, and service level, $S_s$. But offered quality of service $S_s$ depends on $\lambda_s$ as $S_s\geq W_s(\lambda_s, \beta)$. Hence, there is an interaction between the service provider and user set and we are interested in the equilibrium in this  interaction.

Strategy sets of one player depends on the other and hence they are coupled as in constraint games (See \cite{nguyen2013constraint}). Service provider's strategy $\beta$ is constrained by its dependence on strategy of secondary class customers, $\lambda_s$, (say, constraint  (\ref{Sec_Qos})). Similarly, the strategy set of the secondary class customers is constrained by strategies of the service provider, constraint (\ref{Dem}). User-set's strategy space is $R_+$ for the arrival rate $\lambda_s$.  

\subsection{Continuum of NE } \label{game}
  In above two player non-cooperative game, we argue that a continuum of NE exists and discuss its computation via a constrained best response dynamics. 
  
For the given strategy $\bar{\lambda}_s$ of player 2, the best response for player 1 can be obtained by following optimization problem $O1$. 
\begin{eqnarray}
\mbox{\textbf{O1:}\space}{\max_{\theta, S_s, \beta}~} \theta\bar{\lambda}_s
\end{eqnarray}
\begin{eqnarray}Subject~to:~~ 
W_p(\bar{\lambda}_s,\beta)\leq S_p, \\ \label{primary_wtime_constraint}
W_s(\bar{\lambda}_s,\beta)\leq S_s, \\ 
\bar{\lambda}_s \leq \mu - \lambda_p, \\\label{linear_market}
\bar{\lambda}_s \leq a-b\theta-cS_s,\\\nonumber
\theta,S_s\geq 0 ,\beta\in \mathbb{R}^+\cup \infty.
\end{eqnarray}


For the given strategy $(\bar{\theta}$, $\bar{S_s}, \bar{\beta})$ of player 1, the best response of player 2 can be obtained by the trivial optimization problem $O2$:
\begin{eqnarray}\nonumber
\mbox{\textbf{O2:}\space}{\max_{\lambda_s}~} \lambda_s 
\end{eqnarray}
 \begin{equation}\label{demand}
 Subject~to:~~ \lambda_s \leq a- b\bar{\theta} - c \bar{S}_s.
  \end{equation}
It can be argued that the above best-response dynamics converges in one step as Equation (\ref{linear_market}) and (\ref{demand}) hold with equality for either player. It can be easily seen that Equation (\ref{demand}) hold with equality for player 2 due to maximization nature of objective function. Further, Equation (\ref{linear_market}) will hold with equality as $W_p(\bar{\lambda}_s, \beta)$ is an increasing concave function of $\beta\ge 0$ and $W_s(\bar{\lambda}_s, \beta)$ is a decreasing convex function of $\beta\ge 0$. Thus, a continuum of NE exists for any feasible $\bar{\lambda}_s$ which can be easily computed using above best-response dynamics. 

 \begin{rmk}
One of the key insights of the best response dynamics approach is as follows. One can extend the model to non-linear demand functions (such as Cobb-Douglas) using the above best response dynamics. On the other hand, such a generalization would need sophisticated non-linear optimization problem solving scheme if it is approached from standard optimization methods (by formulating problems P1 and P2 as in Section \ref{model_des}).  
\end{rmk}
 
 \subsection{Revenue maximal NE} \label{revenue_max_NE}
 
In above continuum of NE, it is desirable to find the  revenue maximal NE for a given game. Following result characterizes the revenue maximal NE and it is indeed given by the finite step algorithm of Section \ref{algo_validity}.

\begin{thm}
$(\lambda_s^*, \theta^*, S_s^*, \beta^*)$ is revenue maximal NE for above non-cooperative two player game. 
\end{thm}
\begin{proof}
The optimization problem O1 is same as optimization problem P0 with fixed $\lambda_s^*$ and hence its feasible space is a subset of that of P0. And $(\lambda_s^*, \theta^*, S_s^*, \beta^*)$ is a feasible point for problem O1. Thus, $(\theta^*, S_s^*, \beta^*)$   is the optimal solution for problem O1 as we know that optimal solution of P0 is unique from section \ref{UniqueP0}. Hence, $(\theta^*, S_s^*, \beta^*)$ is the best response of player 1 for player 2's strategy  $\lambda_s^*$. The optimal scheduling variable $\beta^*$ could be such that $\beta^* < \infty$ or $\beta^* = \infty$ as determined by the comparison  of optimization problems P1 and P2. 
\end{proof}

Above result follows from the constraint best response optimization problems O1 and O2. 

\subsection{Pricing and revenue management}
Revenue management deals with offering different prices based on customer segments. The most common examples are pricing ``perishable products" and capacity management in airline tickets, high-tech products such as computers and cell phones, high fashion apparel, hospitality industry, fruits and vegetables (see \cite{talluri2006theory}). We present a revenue management scheme via revenue maximal NE in above strategic interaction between service-provider and user-set, in the context of potentially nonperishable items.

We illustrate the computation of revenue maximal NE with a synthetic data set via finite step algorithm for different games. We show that both the service-provider and user-set can benefit at a different revenue maximal NE by playing a different game with marginal inconvenience to primary class of customers.
\begin{table}[h]
\centering{
\begin{tabular}{|l|l|l|l|l|l|}
\hline
\textbf{$S_p$} & \textbf{Priority $\beta^*$} & \textbf{Arrival rate $\lambda_s^*$} & \textbf{Price $\theta^*$} & \textbf{Assured SL $S_s^*$} & \textbf{Revenue $O^*$} \\ \hline
0.29 & 0 & 0.1775 & 1196.4 & 0.5977 & 212.36 \\ \hline
0.35 & 0 & 1.4556 & 1182.7 & 0.9242 & 1721.54 \\ \hline
0.45 & 0 & 3.5858 & 1157.4 & 2.23 & 4150.20 \\ \hline
0.75 & 0.011 & 5.6655 & 1085 & 19.432 & 6147.07 \\ \hline
1 & 0.025 & 5.6655 & 1085 & 19.1672 & 6147.07 \\ \hline
8 & 0.6715 & 5.6655 & 1108.1 & 11.754 & 6277.94 \\ \hline
9.823 & 1 & 5.6655 & 1113.9 & 9.8233 & 6310.80 \\ \hline
10 & 1.0389 & 5.6655 & 1144 & 9.6359 & 6481.33 \\ \hline
12 & 1.624 & 5.6655 & 1208 & 7.5178 & 6843.92 \\ \hline
19 & $\infty$ & 5.6719 & 1141.7 & 0.5195 & 6475.61 \\ \hline
23 & $\infty$ & 5.72 & 1141.2 & 0.5264 & 6527.66 \\ \hline
32 & $\infty$ & 5.799 & 1140.4 & 0.5359 & 6613.18 \\ \hline
\end{tabular}}
\caption{Parameter setting: $a = 120, b =0.1, c = 0.3$, $\lambda_p = 6$, $\mu = 12$, $\sigma = 0.2$.  }
\label{my-label}
\end{table}
For the model parameters as given in Table 1, we compute revenue maximal NE for different primary class customers threshold $S_p$. We have an interesting observation: By marginally increasing the mean waiting time of the primary class customers $S_p$ from $9.823$ to $10$, the service-provider can quote a shorter mean waiting of $9.6359$ as opposed to $9.823$ for an increase in price form $1113.9$ to $1144$. This was possible by giving higher dynamic priority to them as $\beta^*$ is increased to $1.0389$ from equal priority $1$. This new operating point is also revenue maximal NE (for a game with new $S_p$) and results in extra revenue to service-provider as well as shorter mean waiting times to secondary class customers. 

\begin{rmk}
Another managerial insight from this works is that there is no need to offer service to secondary customers or offer in very low volumes when $S_p$ is low. Additionally, for the middle range of service level, the model gives pure dynamic priority scheduling parameter ($0<\beta<\infty$) which should be optimal. 
\end{rmk}
\section{Discussion}\label{discussion}
  
We proposed a queue based game model to capture the strategic interaction between waiting time sensitive user-set and revenue oriented service-provider. We show that a continuum of NE exists and the revenue maximal NE in this continuum can be easily computed via a finite step algorithm, which is now valid as we settled the associated conjecture for the algorithm. Several interesting game theoretic issues are under investigation such as impact of non-linear demand function in strategic interaction, intervention or mechanism design for the market so that it reaches the revenue maximal NE.


}

\bibliographystyle{plain}
\bibliography{EwoRefer}

\end{document}